\documentclass[a4paper, 11pt]{amsart}
\usepackage[utf8]{inputenc}
\usepackage[english]{babel}
\usepackage[T1]{fontenc}
\usepackage{amsmath, amssymb, amsfonts}
\usepackage{enumerate}
\usepackage{graphicx}
\usepackage{comment}
\usepackage[hmargin=1.5cm, vmargin=2.5cm]{geometry}
\usepackage{tikz}
\usepackage{MnSymbol}
\usepackage[colorlinks=true, citecolor=blue]{hyperref}

\title{Discrete quantum subgroups of free unitary quantum groups}
\author{Amaury Freslon}
\email{amaury.freslon@universite-paris-saclay.fr}
\address{Laboratoire de Mathématique d’Orsay, CNRS, Université Paris-Saclay, 91405 Orsay, France.}
\author{Moritz Weber}
\email{weber@math.uni-sb.de}
\address{Saarland University, Fachbereich Mathematik, Postfach 151159, 66041 Saarbr\"ucken, Germany}
\date{}
\thanks{The first author has been supported by the ANR grant ``Noncommutative analysis on groups and quantum groups'' (ANR-19-CE40-0002). The second author has been supported by the SFB-TRR 195 and this work is a contribution to the SFB-TRR 195. Also, he has been supported by the Heisenberg program of the DFG and OPUS-LAP \emph{Quantum groups, graphs and symmetries via representation theory}.}

\theoremstyle{plain}

\newtheorem{thm}{Theorem}[section]
\newtheorem{lem}[thm]{Lemma}
\newtheorem{prop}[thm]{Proposition}
\newtheorem{cor}[thm]{Corollary}

\theoremstyle{definition}
\newtheorem{de}[thm]{Definition}
\newtheorem{ex}[thm]{Example}
\newtheorem{rem}[thm]{Remark}

\DeclareMathOperator{\Irr}{Irr}
\DeclareMathOperator{\id}{id}
\DeclareMathOperator{\Mor}{Mor}
\DeclareMathOperator{\Proj}{Proj}
\DeclareMathOperator{\QA}{Aut^{+}(M_{N}(\C))}
\DeclareMathOperator{\QAut}{Aut^{+}}
\DeclareMathOperator{\Rep}{Rep}
\DeclareMathOperator{\rl}{rl}

\DeclareMathOperator{\cb}{c}

\newcommand{\C}{\mathbb{C}}
\newcommand{\D}{\Delta}
\newcommand{\CC}{\mathcal{C}}

\newcommand{\G}{\mathbb{G}}
\newcommand{\HH}{\mathbb{H}}
\newcommand{\KK}{\mathbb{K}}
\newcommand{\N}{\mathbf{N}}
\newcommand{\NCc}{NC^{\circ\bullet}}
\newcommand{\PP}{\mathcal{P}}

\newcommand{\Pc}{P^{\circ\bullet}}

\newcommand{\Z}{\mathbf{Z}}

\renewcommand{\O}{\mathcal{O}}

\begin{document}

\begin{abstract}
We classify all  compact quantum groups whose C*-algebra sits inside that of the free unitary quantum groups $U_{N}^{+}$. In other words, we classify all discrete quantum subgroups of $\widehat{U}_{N}^{+}$, thereby proving a quantum variant of Kurosh's theorem to some extent. This yields interesting families which can be described using free wreath products and free complexifications. They can also be seen as quantum automorphism groups of specific quantum graphs which generalize finite rooted regular trees, providing explicit examples of quantum trees.
\end{abstract}

\maketitle

\section{Introduction}

It is one of the most natural questions, once an algebraic structure is defined, to try to understand its sub-objects. For instance, as soon as finite groups are defined, one tries to classify all  subgroups of a given one. However, for infinite groups the question quickly becomes very complicated, as the following fact shows: given an integer $N$, any group with $N$ generators gives rise to a quotient of the free group $\mathbb{F}_{N}$, which then also yields a  normal subgroup of $\mathbb{F}_{N}$. Classifying all subgroups of $\mathbb{F}_{N}$ therefore seems an untractable problem. Nevertheless, a lot can be said about the structure of these subgroups: by a theorem of Kurosh, they are all free groups.

The purpose of this work is to consider similar questions in the setting of discrete quantum groups. It has long been recognized that the free unitary quantum groups $U_{N}^{+}$, see \cite{wang1995free}, or more precisely their dual discrete quantum groups $\widehat{U}_{N}^{+}$, play a role in the quantum theory analogous to that of free groups; see for instance \cite{banica2007note, banica1997groupe, vaes2007boundary, vergnioux2013k} and many more articles. We therefore simply ask the following question: what are the discrete quantum subgroups of $\widehat{U}_{N}^{+}$, and what can be said about them? The notion of a free quantum group is not rigorously defined in the setting of discrete quantum groups, but there is a wealth of examples sharing some kind of freeness properties which are based on the combinatorics of non-crossing partitions (or on free fusion semi-rings). Our first result is that discrete quantum subgroups of $\widehat{U}_{N}^{+}$ all belong to that class, which is, in a way, a quantum analogue of Kurosh's theorem. However, note a subtle difference: In Kurosh's theorem, all subgroups of $\mathbb{F}_{N}$ are free groups --  and vice versa, all free groups are subgroups of $\mathbb{F}_{N}$. In the quantum setting, we only have that all discrete quantum subgroups of $\widehat{U}_{N}^{+}$ have free fusion semi-rings, but not all free fusion semi-rings coming from quantum groups arise in this way. In any case, we can completely classify all discrete quantum subgroups of $\widehat{U}_{N}^{+}$ and describe them explicitely.

As it turns out, we will work these result backwards. First, we will make use of the rich combinatorial structure underlying $U_{N}^{+}$, which involves non-crossing partitions, to give an abstract classification result for discrete quantum subgroups of $\widehat{U}_{N}^{+}$. In this context, we will also make use of the theory of easy quantum groups, see \cite{banica2009liberation, weber2015introduction, weber2012classification, raum2013full, freslon2023compact} amongst others.
The key tool for our classification is the notion of \emph{module of projection partitions}, introduced in \cite{freslon2021tannaka} for a completely different purpose. We prove that this notion exactly captures dual quantum subgroups, and admits a simple combinatorial description which is easier to work with. Based on this, we will then proceed to give an explicit description of all the quantum subgroups that we found. It is a rather surprising consequence of our work that they can be described from the quantum automorphism group of $M_{N}(\C)$ using two standard quantum group constructions: the free wreath product and the free complexification. Also, it links with recent work on (quantum) graphs \cite{wasilewski2023quantum, brownlowe2023self}, as we will show, allowing for a definition of quantum rooted regular trees.

Let us now outline the contents of the paper. Section \ref{sec:preliminaries} contains preliminary material concerning partitions, compact quantum groups and easy quantum groups, mainly to fix notations and terminology. Note that for practicality, we take the compact point of view and therefore focus on the notion of ``dual quantum subgroup'' (see Definition \ref{de:dualquantumsubgroup}) instead of talking about discrete quantum subgroups of the dual, even though this is equivalent. Then, Section \ref{sec:correspondence} makes the connection between dual quantum subgroups and the modules of projective partitions introduced in \cite{freslon2021tannaka}, our first main result being Theorem \ref{thm:correspondence}.
As an illustration of the power of modules of projective partitions, we first classify in Section \ref{sec:orthogonal} the dual quantum subgroups of orthogonal easy quantum groups, see the overview on page \pageref{table:overvieworthogonalmodules}. Once this is done, we turn in Section \ref{sec:unitary} to the heart of this work and classify all dual quantum subgroups of the free unitary quantum groups, see Theorem \ref{thm:classificationunitary}. In Section \ref{sec:quantumsubgroups}, we then provide an explicit description of all discrete quantum subgroups of $\widehat{U_N^+}$ involving free wreath products and free complexifications, our second main result being Theorem \ref{thm:descriptionquantumsubgroups}.
Subsequently, we define (rooted tracial $M_N(\mathbb C)$-regular) quantum trees in Section \ref{sec:quantumtrees} and we conclude with yet another description of the main one-parameter family of dual quantum subgroups of $U_{N}^{+}$ as quantum automorphism groups of such rooted quantum trees, see Theorem \ref{thm:quantumtrees}.

\section{Preliminaries}\label{sec:preliminaries}

In this preliminary section we will recall the basics concerning easy quantum groups first defined by Banica and Speicher \cite{banica2009liberation}, see also \cite{weber2015introduction, freslon2023compact}, starting with the notion of a partition which is at the heart of the theory.

\subsection{Colored partitions}

Easy quantum groups are based on the combinatorics of partitions, and in particular noncrossing ones. A \emph{partition} consists in two integers $k$ and $l$, and a partition of the set $\{1, \dots, k+l\}$. We think of it as an upper row of $k$ points, a lower row of $l$ points and some strings connecting these points. If the strings may be drawn in such a way that they do not cross, the partition will be said to be \emph{noncrossing}. The set of all partitions is denoted by $P$ and the set of all noncrossing partitions is denoted by $NC$.

A maximal set of points which are all connected in a partition is called a \emph{block}. We denote by $b(p)$ the number of blocks of a partition $p$, by $t(p)$ the number of \emph{through-blocks}, i.e. blocks containing both upper and lower points, and by $\beta(p) = b(p)-t(p)$ the number of \emph{non-through-blocks}. This work is concerned with a refinement of the notion of partitions: colored partitions.

\begin{de}
A \emph{(two-)colored partition} is a partition with the additional data of a color (black or white) for each point. The set of all colored partitions is denoted by $\Pc$ and the set of noncrossing colored partitions is denoted by $\NCc$.
\end{de}

In the example below, $p_{1}$ has crossings while $p_{2}$ is a noncrossing colored partition.

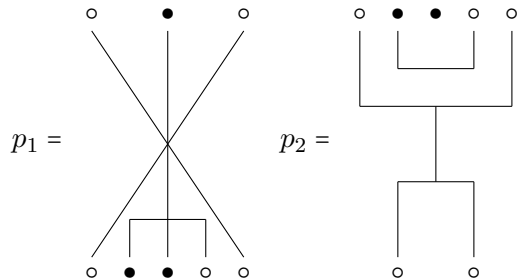
\begin{center}
\begin{tikzpicture}[scale=0.5]
\draw (0,-3) -- (0,3);
\draw (-1,-3) -- (-1,-2);
\draw (1,-3) -- (1,-2);
\draw (-1,-2) -- (1,-2);
\draw (-2,-3) -- (2,3);
\draw (2,-3) -- (-2,3);

\draw (-2,3) node[above]{$\circ$};
\draw (0,3) node[above]{$\bullet$};
\draw (2,3) node[above]{$\circ$};

\draw (-2,-3) node[below]{$\circ$};
\draw (-1,-3) node[below]{$\bullet$};
\draw (0,-3) node[below]{$\bullet$};
\draw (2,-3) node[below]{$\circ$};
\draw (1,-3) node[below]{$\circ$};

\draw (-2.5,0) node[left]{$p_{1} = $};
\end{tikzpicture}
\begin{tikzpicture}[scale=0.5]
\draw (0,-1) -- (0,1);
\draw (-2,1) -- (2,1);
\draw (-2,1) -- (-2,3);
\draw (2,1) -- (2,3);
\draw (-1,2) -- (-1,3);
\draw (1,2) -- (1,3);
\draw (-1,2) -- (1,2);

\draw (-1,-1) -- (1,-1);
\draw (-1,-1) -- (-1,-3);
\draw (1,-1) -- (1,-3);

\draw (-2,3) node[above]{$\circ$};
\draw (-1,3) node[above]{$\bullet$};
\draw (0,3) node[above]{$\bullet$};
\draw (2,3) node[above]{$\circ$};
\draw (1,3) node[above]{$\circ$};

\draw (-1,-3) node[below]{$\circ$};
\draw (1,-3) node[below]{$\circ$};

\draw (-2.5,0) node[left]{$p_{2} = $};
\end{tikzpicture}
\end{center}

From now on, the word ``partition'' will always mean ``two-colored partition''. Partitions can be combined using the following \emph{category operations}:
\begin{itemize}
\item  If $p\in \Pc(k, l)$ and $q\in \Pc(k', l')$, then $p\odot q\in \Pc(k+k', l+l')$ is their \emph{horizontal concatenation}, i.e. the first $k$ of the $k+k'$ upper points are connected by $p$ to the first $l$ of the $l+l'$ lower points, whereas $q$ connects the remaining $k'$ upper points with the remaining $l'$ lower points.
\item If $p\in \Pc(k, l)$ and $q\in \Pc(l, m)$ are such that the coloring of the lower row of $p$ is the same as the coloring of the upper row of $q$, then $qp\in \Pc(k, m)$ is their \emph{vertical concatenation}, i.e. $k$ upper points are connected by $p$ to $l$ middle points and the lines are then continued by $q$ to $m$ lower points. This process may produce loops in the partition. More precisely, consider the set $L$ of elements in $\{1, \dots, l\}$ which are not connected to an upper point of $p$ nor to a lower point of $q$. The lower row of $p$ and the upper row of $q$ both induce partitions of the set $L$. The maximum (with respect to inclusion) of these two partitions is the \emph{loop partition} of $L$, its blocks are called \emph{loops} and their number is denoted by $\rl(q, p)$. To complete the operation, we remove all the loops.
\item If $p\in \Pc(k, l)$, then $p^{*}\in \Pc(l, k)$ is the partition obtained by reflecting $p$ with respect to the horizontal axis (without changing the colors).
\item If $p\in \Pc(k, l)$, then we can shift the very left upper point to the left of the lower row (or the converse) and change its color. We do not change the strings connecting the points in this process. This gives rise to a partition in $\Pc(k-1, l+1)$ (or in $\Pc(k+1, l-1)$), called a \emph{rotated version} of $p$. We can also rotate partitions on the right.
\item Using, the category operations above, one can \emph{reverse} a partition $p$  by rotating all its upper points to the lower row and all its lower points to the upper row. This gives a new partition $\overline{p}$. Equivalently, $\overline{p}$ is obtained by reflecting $p$ with respect to a vertical axis and then inverting the colors.
\end{itemize}

As an example, we give the vertical concatenation of the two partitions $p_{1}$ and $p_{2}$ defined above:
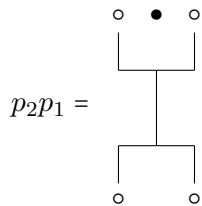
\begin{center}
\begin{tikzpicture}[scale=0.5]
\draw (0,-1) -- (0,1);
\draw (-1,1) -- (1,1);
\draw (-1,1) -- (-1,2);
\draw (1,1) -- (1,2);

\draw (-1,-1) -- (1,-1);
\draw (-1,-1) -- (-1,-2);
\draw (1,-1) -- (1,-2);

\draw (-1,2) node[above]{$\circ$};
\draw (0,2) node[above]{$\bullet$};
\draw (1,2) node[above]{$\circ$};

\draw (-1,-2) node[below]{$\circ$};
\draw (1,-2) node[below]{$\circ$};

\draw (-1.5,0) node[left]{$p_{2}p_{1} = $};
\end{tikzpicture}
\end{center}

There are four ways of coloring the partition $\vert\in \Pc(1, 1)$. If the two points are white (resp. black), we will call it the \emph{white identity} (resp. \emph{black identity}) partition and denote it by $\id_{\circ}$ (resp. $\id_{\bullet}$). Note that these two partitions are rotated versions of one another.

\begin{de}\label{defcategorypartitions}
A \emph{category of partitions} $\CC$ is the data of a collection of sets $\CC(k, l)$ of colored partitions, for all integers $k$ and $l$, which is stable under the above category operations and contains the white identity partition (hence also the black identity partition).
\end{de}

Let us further introduce two specific partitions which will play an important role hereafter. We define the \emph{duality partitions} as
\begin{center}
\begin{tikzpicture}[scale=0.5]
\draw (-1, 0)--(-1,-1);
\draw (1,0)--(1,-1);
\draw (-1,-1)--(1,-1);

\draw (-1,0)node[above]{$\circ$};
\draw (1,0)node[above]{$\bullet$};
\draw (-1.25, -0.25)node[left]{$D_{\circ\bullet} = $};
\end{tikzpicture}
\text{ and }
\begin{tikzpicture}[scale=0.5]
\draw (-1, 0)--(-1,-1);
\draw (1,0)--(1,-1);
\draw (-1,-1)--(1,-1);

\draw (-1,0)node[above]{$\bullet$};
\draw (1,0)node[above]{$\circ$};
\draw (-1.25, -0.25)node[left]{$D_{\bullet\circ} = $};
\end{tikzpicture}
\end{center}
The crucial notion for the study of the representation theory of easy quantum groups is that of a \emph{projective partition}.

\begin{de}\label{de:projectivepartitions}
A partition $p\in \Pc(k, k)$ is said to be \emph{projective} if it satisfies $pp = p = p^{*}$. We denote by $\Proj_{\CC}$ the set of projective partitions in a category of partitions $\CC$.
\end{de}

There are actually many of them, according to the following result (see \cite[Prop 2.12]{freslon2013representation}):

\begin{prop}\label{prop:characterizationprojective}
A partition $p\in \Pc(k, k)$ is projective if and only if there exists a partition $r\in \Pc(k, k)$ such that $r^{*}r = p$.
\end{prop}

\subsection{Easy quantum groups}

Easy quantum groups are a family of compact quantum groups which can be constructed using categories of partitions. To explain how, we will first explain how partitions can give rise to linear maps between Hilbert spaces, and then briefly recall some notions on compact quantum groups before turning to the main definition.

\subsubsection{Partitions and linear maps}

The link between partitions and quantum groups lies in the following definition \cite[Def 1.6]{banica2009liberation}. Note that this definition does not involve the coloring of the partitions.

\begin{de}
Let $N$ be an integer and let $(e_{1}, \dots, e_{N})$ be the canonical basis of $\C^{N}$. For any partition $p\in \Pc(k, l)$, we define a linear map
\begin{equation*}
T_{p}:(\C^{N})^{\otimes k} \mapsto (\C^{N})^{\otimes l}
\end{equation*}
by the following formula:
\begin{equation*}
T_{p}(e_{i_{1}} \otimes \dots \otimes e_{i_{k}}) = \sum_{j_{1}, \dots, j_{l} = 1}^{n} \delta_{p}(i, j)e_{j_{1}} \otimes \dots \otimes e_{j_{l}},
\end{equation*}
where $\delta_{p}(i, j) = 1$ if and only if as soon as two points are connected by a string of the partition $p$, the corresponding indices of the multi-index $i = (i_{1}, \dots, i_{k})$ are equal. Otherwise, $\delta_{p}(i, j) = 0$.
\end{de}

The interplay between these maps and the category operations are given by the following rules proven in \cite[Prop. 1.9]{banica2009liberation}:

\begin{itemize}
\item $T_{p}^{*} = T_{p^{*}}$;
\item $T_{p}\otimes T_{q} = T_{p\otimes q}$;
\item $T_{pq} = N^{\rl(p, q)} T_{p}\circ T_{q}$.
\end{itemize}

In particular, by Proposition \ref{prop:characterizationprojective}, $T_{r}$ is always a multiple of a partial isometry, which is an orthogonal projection if and only if $r$ is projective. It should be stressed that the maps $T_{p}$ are not linearly independent in general. However, restricting to the noncrossing case rules out this problem (see e.g. \cite[Lem 4.16]{freslon2013representation} for a proof).

\begin{prop}\label{prop:linearindependence}
Let $\CC$ be a category of \emph{noncrossing} partitions and fix an integer $N\geqslant 4$. Then, for any $w, w'\in F$, the maps $(T_{p})_{p\in \CC(w, w')}$ are linearly independent.
\end{prop}

\subsubsection{Tannaka-Krein duality and compact quantum groups}

We refer the reader to \cite{timmermann2008invitation} and \cite{neshveyev2014compact} for detailed treatments of the theory of compact quantum groups, and to \cite{freslon2023compact} for an account emphasizing the combinatorial aspects. It is known since the work of M. Dijkhuizen and T. Koornwinder \cite{dijkhuizen1994CQG} that compact quantum groups can be treated algebraically through the following notion of a CQG-algebra:

\begin{de}
A \emph{CQG-algebra} is a Hopf $*$-algebra which is spanned by the coefficients of its finite-dimensional unitary corepresentations.
\end{de}

If $G$ is a compact group, then its algebra of regular functions $\O(G)$ is a CQG-algebra. Based on that example, and in an attempt to retain the intuition coming from the classical setting, we will denote a general CQG-algebra by $\O(\G)$ and say that it corresponds to the \emph{compact quantum group} $\G$. If $\Gamma$ is a discrete group and $\C[\Gamma]$ denotes its group algebra, it is easy to endow it with a Hopf $*$-algebra structure with coproduct given by $\D(g) = g\otimes g$ for all $g\in \Gamma$. Since this turns each $g\in \Gamma\subset\C[\Gamma]$ into a one-dimensional corepresentation, it yields a CQG-algebra. The resulting compact quantum group is called the \emph{dual} of $\Gamma$ and is denoted by $\widehat{\Gamma}$.

The main feature of compact quantum groups is their nice representation theory, which is just another point of view on the corepresentation theory of the corresponding CQG-algebra. Let us give a definition to make this clear.

\begin{de}
An \emph{$n$-dimensional representation} of $\G$ is an element $v\in M_{n}(\O(\G))$ which is invertible and such that for all $1\leqslant i, j\leqslant n$,
\begin{equation*}
\D(v_{ij}) = \sum_{k=1}^{n}v_{ik}\otimes v_{kj}.
\end{equation*}
It is said to be \emph{unitary} if it is unitary as an element of the $*$-algebra $M_{n}(\O(\G))$.
\end{de}

\begin{ex}
Let $\mathbf{1}_{\G}$ denote the unit of $\O(\G)$ seen as an element of $M_{1}(\O(\G))$. It satisfies $\D(\mathbf{1}_{\G}) = \mathbf{1}_{\G}\otimes \mathbf{1}_{\G}$, hence is a (unitary) representation, called the \emph{trivial representation of $\G$}.
\end{ex}

Given two representations $v$ and $w$, one can form their \emph{direct sum} by considering a block diagonal matrix with blocks $v$ and $w$ respectively, and their \emph{tensor product} by considering the matrix with coefficients
\begin{equation*}
(v\otimes w)_{(i,k),(j, l)} = v_{ij}w_{kl}.
\end{equation*}
Moreover, the \emph{conjugate} of $v$ is the representation $\overline{v}$ with coefficients $(v_{ij}^{*})_{1\leqslant i, j\leqslant n}$. In this setting, an intertwiner between two representations $v$ and $w$ of dimension $n$ and $m$ respectively will be a linear map $T : \C^{n}\to \C^{m}$ such that $Tv = wT$ (we are here identifying $M_{n}(\C)$ with $M_{n}(\C.1_{\O(\G)})\subset M_{n}(\O(\G))$). The set of all intertwiners between $v$ and $w$ will be denoted by $\displaystyle\Mor_{\G}(v, w)$.

Two representations are said to be \emph{equivalent} if there exists an invertible intertwiner between them. If $T$ is injective, then $v$ is said to be a \emph{subrepresentation} of $w$, and if $w$ admits no subrepresentation apart from itself, then it is said to be \emph{irreducible}. One of the fundamental results in the representation theory of compact quantum groups is due to S.L. Woronowicz in \cite{woronowicz1995compact} and can be summarized as follows.

\begin{thm}[Woronowicz]
Any finite-dimensional representation of a compact quantum group splits as a direct sum of irreducible ones, and any irreducible representation is equivalent to a unitary one.
\end{thm}

Let us now consider a compact quantum group $\G$ with a \emph{fundamental representation}, i.e. a finite-dimensional representation $u$ whose coefficients generate $\O(\G)$. This implies that any irreducible representation of $\G$ arises as a subrepresentation of some tensor product of $u$ and its conjugate $\overline{u}$. Let us associate to any word $w = w_{1}\dots w_{k}$ in the free monoid $F$ over $\{\circ, \bullet\}$ a representation $u^{\otimes w}$ by setting
\begin{equation*}
u^{\otimes w} = u^{w_{1}}\otimes \dots \otimes u^{w_{k}},
\end{equation*}
where by convention $u^{\circ} = u$ and $u^{\bullet} = \overline{u}$. We can now define a categorical object which encapsulates the whole quantum group.

\begin{de}
The \emph{category of representations} of $(\G, u)$ is the rigid C*-tensor category (see \cite[Def 2.1.1]{neshveyev2014compact} for the definition) $\Rep(\G)$ with object set $F$ and morphism spaces
\begin{equation*}
\Mor_{\Rep(\G)}(w, w') = \Mor_{\G}(w, w').
\end{equation*} 
\end{de}

Conversely, given a family $(\Mor(w, w'))_{w, w'}$ of finite-dimensional vector spaces with sufficiently nice properties, one can reconstruct a compact quantum group $\G$ using S.L. Woronowicz's Tannaka-Krein theorem \cite[Thm 1.3]{woronowicz1988tannaka}. Let us state this theorem in the particular case which is relevant for us, going back to Banica and Speicher in this particular version \cite{banica2009liberation}. If $\CC$ is a category of partitions and if $w, w'\in F$, we will denote by $\CC(w, w')$ the set of partitions $p\in \CC(\vert w\vert, \vert w'\vert)$ such that the upper coloring of $p$ is $w$ and the lower coloring of $p$ is $w'$ (here $\vert w\vert$ denotes the length of the word $w$).

\begin{thm}[Woronowicz, Banica-Speicher]
Let $\CC$ be a category of partitions and let $N$ be an integer. Then, there exists a unique (up to isomorphism) pair $(\G, u)$, where $\G$ is a compact quantum group and $u$ is a fundamental representation of $\G$ such that $\Mor(u^{\otimes w}, u^{\otimes w'})$ is the linear span of the maps $T_{p}$ for $p\in \CC(w, w')$. Such a $\G=\G_N(\CC)$ will be called a \emph{(unitary) easy quantum group} or a \emph{partition quantum group}.
\end{thm}

 Let $\CC_{U}$ be the smallest category of partitions (i.e. the one generated by the white identity partition). The associated quantum group is the \emph{free unitary quantum group} $U_{N}^{+}$ introduced by S. Wang in \cite{wang1995free}. Since inclusion of categories of partitions translates into quotients of the corresponding CQG-algebras, and since quotients of $\O(G)$ for a compact group of matrices $G$ correspond to closed subgroups of $G$, we see that any easy quantum group is a \emph{closed quantum subgroup} of $U_{N}^{+}$. The other extreme case is the category of all partitions $\Pc$, which yields the symmetric group $S_{N}$. Thus, easy quantum groups form a special class of quantum groups $\G$ in the range
\begin{equation*}
S_{N} \subset \G \subset U_{N}^{+}.
\end{equation*}
Other examples of easy quantum groups include S. Wang's free quantum permutation group $S_{N}^{+}$ ($\CC = \NCc$) and free orthogonal quantum group $O_{N}^{+}$ ($\CC =$ all partitions with blocks of size $2$). We refer the reader to \cite{wang1995free} and \cite{wang1998quantum} for the definition of these quantum groups and to \cite{banica2009liberation} for proofs of these facts and to \cite{raum2013full, tarrago2018classification, tarrago2015unitary, weber2012classification} for more on the classification of these objects.

We conclude with a description of the representation theory of $U_{N}^{+}$, since this will be central in this work. Let $M$ be the free monoid over the set $\{\circ, \bullet\}$, that is, the set of all words on these two letters. We define a map $w\mapsto \overline{w}$ from $M$ to $M$ by $\overline{\circ} = \bullet$, $\overline{\bullet} = \circ$ and for a word $w = w_{1}\cdots w_{n}$, $\overline{w} = \overline{w}_{n}\cdots\overline{w}_{1}$. Then, according to \cite{banica1997groupe}, the irreducible representations of $U_{N}^{+}$ can be indexed by the elements of $M$ in such a way that $u^{\circ} = u$, $\overline{u^{w}} = u^{\overline{w}}$ and
\begin{equation*}
u^{w}\otimes u^{w'} = \sum_{w = az, w' = \overline{z}b}u^{ab}.
\end{equation*}

\section{Modules of projective partitions}\label{sec:correspondence}

\subsection{Definition and first properties}

This work is concerned with the following combinatorial object, introduced in \cite{freslon2021tannaka}. Recall the definition of projective partitions from Definition \ref{de:projectivepartitions} and the operations from Definition \ref{defcategorypartitions}.

\begin{de}
Let $\CC$ be a category of partitions. A \emph{module of projective partitions} over $\CC$ is a set $\PP$ of projective partitions such that
\begin{enumerate}
\item For any $p, q\in \PP$, $p\odot q\in \PP$;
\item For any $p\in \PP$, $\overline{p}\in \PP$;
\item For any $p\in \PP$ and $r\in \CC$, $rpr^{*}\in \PP$ as soon as the composition makes sense.
\end{enumerate}
\end{de}

The definition above is the original one and was established with a view towards the construction of weak unitary tensor functors on $\Rep(\G_{N}(\CC))$ (see \cite[Sec 4]{freslon2021tannaka} for details). Nevertheless, we will in the present work be interested in another application of modules of projective partitions for which an alternate characterization will prove useful. To state it, we need some notions on projective partitions from \cite[Sec 2.5]{freslon2013representation}.

\begin{de}
Let $p$ and $q$ be two projective partitions. We say that $p$ is \emph{equivalent} to $q$ in $\CC$, and write $p\sim_{\CC} q$, if there exists $r\in \CC$ such that $p = r^{*}r$ and $q = rr^{*}$. Moreover, we say that $q$ is \emph{dominated} by $p$, and write $q\prec p$, if $qp = q = pq$.
\end{de}

We are now ready for some elementary but very useful properties of modules of projective partitions.

\begin{prop}\label{prop:comparison}
Let $\PP\subset \Proj_{\CC}$ be a module of projective partitions over $\CC$. Then, for any $p\in \PP$ and $q\in \Proj_{\CC}$,
\begin{enumerate}
\item If $p\sim_{\CC} q$, then $q\in \PP$;
\item If $q\prec p$, then $q\in \PP$.
\end{enumerate}
Conversely, let $\PP\subset \Proj_{\CC}$ be such that $\PP\odot \PP\subset \PP$, $\overline{\PP} = \PP$ and the two properties above hold. Then, $\PP$ is a module of projective partitions over $\CC$.
\end{prop}

\begin{proof}
Let $p, q\in \PP$.
\begin{enumerate}
\item If $p\sim q$, consider $r\in \CC$ such that $r^{*}r = p$ and $rr^{*} = q$. Then,
\begin{equation*}
q = qq = rr^{*}rr^{*} = rpr^{*}\in \PP.
\end{equation*}
\item Observe that if $q\prec p$, then $qp = q = pq$ hence
\begin{equation*}
q = qq = (qp)(pq) = qpq\in \PP
\end{equation*}
by definition.
\end{enumerate}
To prove the second claim, it is enough to prove that for $r\in \CC$ and $p\in \PP$, we obtain $rpr^{*}\in \PP$ if the composition makes sense. This follows from the fact that $rpr^{*} = (rp)(rp)^{*}$ is equivalent in $\CC$ (the assumption $\PP\subset \Proj_{\CC}$ ensures that $rp\in \CC$) to $(rp)^{*}(rp) = pr^{*}rp$, which is dominated by $p$.
\end{proof}

For convenience and since this is the case of interest for us, we will say that a module of projective partitions over $\CC$ is \emph{standard} if it is contained in $\Proj_{\CC}$ and we will only consider standard modules in the sequel.

\begin{ex}\label{ex:trivial}
For any category of partitions $\CC$, $\Proj_{\CC}$ is of course a standard module of projective partitions over $\CC$. Moreover, let $\Proj_{\CC}^{0}$ be the set of all projective partitions with no through-block. Then, this is also a standard module of projective partitions over $\CC$.
\end{ex}

\subsection{Link with dual quantum subgroups}

Our aim is now to describe a correspondence between standard modules of projective partitions over a category of partitions $\CC$ and quantum groups associated to $\G_{N}(\CC)$. Let us introduce a bit of terminology for that purpose.

\begin{de}\label{de:dualquantumsubgroup}
A \emph{dual quantum subgroup} of a compact quantum group $\G$ is a compact quantum group $\HH$ such that $\O(\HH)\subset \O(\G)$ as Hopf $*$-algebra.
\end{de}

\begin{rem}
If $\G = \widehat{\Gamma}$ is dual to a discrete group, then a dual quantum subgroup is of the form $\HH = \widehat{\Lambda}$ for some subgroup $\Lambda < \Gamma$, hence the name. More generally, once a suitable notion of discrete quantum group is introduced together with the appropriate extension of Pontryagin duality, one may define dual quantum subgroups as duals of quantum subgroups of the discrete dual of $\G$.
\end{rem}

The main result of this section is that dual quantum subgroups and standard modules of projective partitions are in one-to-one correspondence. The proof uses the fact that the representation theory of $\G_{N}(\CC)$ can be described, at least in the non-crossing case, purely from the combinatorics of $\CC$. This was done in \cite{freslon2013representation} and we now recall some of the main results of that work. First observe that for any partition $r$, the fact that $r^{*}r$ is projective implies that $T_{r}$ is a multiple of a partial isometry, which we denote by $S_{r}$. Then, for a projective partition $p$, we set
\begin{equation*}
P_{p} = S_{p} - \bigvee_{q\prec p} S_{q},
\end{equation*}
which is a projection in $\Mor_{\G_{N}(\CC)}(w, w)$, where $w$ is the upper colouring of $p$. This yields a subrepresentation $u^{p}\subset u^{w}$ and the following holds \emph{if $\CC$ is non-crossing} (see \cite[Sec 4 and 5]{freslon2013representation}):
\begin{itemize}
\item $u^{p}$ is irreducible for all $p\in \Proj_{\CC}$;
\item Any irreducible representation of $\G_{N}(\CC)$ is equivalent to $u^{p}$ for some $p\in \Proj_{\CC}$;
\item $u^{p}\sim u^{q}$ as representations of $\G_{N}(\CC)$ if and only if $p\sim_{\CC} q$.
\end{itemize}

Let us now consider a dual quantum subgroup $\HH$ of $\G_{N}(\CC)$. As explained in \cite[Lem 2.1]{vergnioux2004k}, associating to $\HH$ its category of representations $\Rep(\HH)\subset \Rep(\G_{N}(\CC))$ yields a one-to-one correspondence between dual quantum subgroups and C*-tensor subcategories of $\Rep(\G_{N}(\CC))$. We may therefore define
\begin{equation*}
\PP(\HH) = \{p\in \Proj_{\CC} \mid u^{p}\in \Rep(\HH)\} \subset \Proj_{\CC}
\end{equation*}
as a natural candidate for a module of projective partitions. Some properties are straightforward to check.

\begin{lem}\label{lem:projectivefromsubgroup}
We have $\PP(\HH)\odot\PP(\HH)\subset \PP(\HH)$ and $\overline{\PP(\HH)} = \PP(\HH)$.
\end{lem}

\begin{proof}
The first inclusion comes from the fact that $u^{p\odot q}$ is a subrepresentation of $u^{p}\otimes u^{q}$ by \cite[Thm 4.27]{freslon2013representation} hence in $\Rep(\HH)$. The second one comes from the fact that $u^{\overline{p}}$ is the conjugate of $u^{p}$, by \cite[Rem 6.12]{freslon2013representation}, hence in $\Rep(\HH)$.
\end{proof}

As for the third property, we will instead use the characterization of Proposition \ref{prop:comparison} and this will require a few facts concerning non-crossing projective partitions. We refer to \cite[Def 5.6]{freslon2013representation} for the definition of the binary operations $\square$ and $\boxvert$ on partitions.

\begin{lem}\label{lem:structuresubpartition}
Let $\CC$ be a category of non-crossing partitions and let $p\in \Proj_{\CC}$. Then, there exists $b_{1}, \cdots, b_{t(p)}\in \Proj_{NC}$ such that $t(b_{i}) = 1$ for all $1\leqslant i\leqslant t(p)$ and
\begin{equation*}
p = b_{1}\odot\cdots\odot b_{t(p)}.
\end{equation*}
Moreover, let $p_{i}^{\square}$ be the partition obtained by replacing $b_{i}\odot b_{i+1}$ by $b_{i}\square b_{i+1}$ and similarly for $p_{i}^{\boxvert}$. Then $q \prec p$ if and only if $q\preceq p_{i}^{\boxvert}$ or $q\preceq p_{i}^{\square}$ for some $1\leqslant i\leqslant t(p)-1$.
\end{lem}

\begin{proof}
The first part of the statement is \cite[Lem 4.4]{freslon2013fusion}. As for the second part, this is simply a rewriting of \cite[Prop 2.23]{freslon2013representation} (note that even though that article is written in the setting of uncoloured partitions, the proofs carry on to the coloured setting verbatim, see also \cite[Sec 3.3]{freslon2014partition}).
\end{proof}

We are now ready to prove our first main result, connecting modules of projective partitions and dual quantum subgroups.

\begin{thm}\label{thm:correspondence}
There is a one-to-one correspondence between dual quantum subgroups of $\G_{N}(\CC)$ and standard modules of projective partitions over $\CC$.
\end{thm}

\begin{proof}
Let $\HH$ be a dual quantum subgroup. If $p\sim_{\CC} q$ and $p\in \PP(\HH)$, then $u^{q}\sim u^{p}$ as representations of $\G_{N}(\CC)$. This means that there exists a linear isomorphism $T$ such that $Tu^{p} = u^{q}T$. In particular, the coefficients of $u^{q}$ are linear combinations of those of $u^{p}$, so that they belong to $\O(\HH)$. This means that $u^{q}$ is in fact a representation of $\HH$, hence $q\in \PP(\HH)$. We will now prove by induction on the number of through-blocks that if $p\in \PP(\HH)$ and $q\prec p$, then $q\in \PP(\HH)$. If $t(p) = 0$ there is nothing to prove, we will therefore assume that the result holds as soon as $t(p)\leqslant n$ for some $n\in \N$ and consider $p\in \PP(\HH)$ with $t(p) = n+1$. By Lemma \ref{lem:structuresubpartition}, it is enough to prove that $p_{i}^{\boxvert},p_{i}^{\square}\in \PP(\HH)$ for all $1\leqslant i\leqslant t(p)$ as soon as these partitions are in $\CC$. To do this, observe first that $u^{p}\otimes u^{\overline{p}}$ contains a subrepresentation equivalent to $u^{s}$, where
\begin{equation*}
s = b_{1}\odot\cdots \odot b_{i}\odot\overline{b}_{i}\cdots \odot\overline{b}_{1}
\end{equation*}
so that $s\in \PP(\HH)$. But now, $u^{s}\otimes u^{p}$ contains a subrepresentation equivalent to $p_{i}^{\boxvert}$ and one equivalent to $p_{i}^{\square}$ (if these are in $\CC$), and the proof is complete.

Conversely, let $\PP\subset \Proj_{\CC}$ be a standard module of projective partitions and consider the full sub-category $\Rep(\PP)$ of $\Rep(\G_{N}(\CC))$ with objects $u^{p}$ for all $p\in \PP$. By \cite[Thm 4.27]{freslon2013representation}, the subrepresentations of $u^{p}\otimes u^{q}$ are all equivalent to some $u^{s}$ with $s\prec p\odot q$, thus by Proposition \ref{prop:comparison} they are all in $\Rep(\PP)$ and this shows that $\Rep(\PP)$ is a tensor category. Because $\overline{\PP} = \PP$, $\Rep(\PP)$ is also a C*-category so that this is the representation category of a dual quantum subgroup $\HH$ of $\G_{N}(\CC)$, and the proof is complete.
\end{proof}

Let illustrate this with the two modules of projective partitions of Example \ref{ex:trivial}:
\begin{itemize}
\item $\Proj_{\CC}$ is by definition equal to $\PP(\G_{N}(\CC))$;
\item It is proven in \cite[Lem 5.1]{freslon2013fusion} that $t(p) = 0$ if and only if $u^{p}$ is one-dimensional. A one-dimensional representation is the same as a \emph{group-like element}, i.e. $x\in C(\G_{N}(\CC))$ such that $\D(x) = x\otimes x$. These form a group $G$ and the $*$-subalgebra that they generate is isomorphic to $\C[G] = \O(\widehat{G})$. Thus, the dual quantum subgroup corresponding to $\Proj_{\CC}^{0}$ of Example \ref{ex:trivial} is $\widehat{G}$.
\end{itemize}

Let us conclude this section with an elementary fact which will be useful later on, which is the simple observation that the correspondance $\HH\leftrightarrow \PP(\HH)$ respects inclusions merely by definition.

\begin{lem}
Let $\HH$ and $\KK$ be dual quantum subgroups of $\G_{N}(\CC)$. Then, $\KK$ is a dual quantum subgroup of $\HH$ if and only if $\PP(\KK)\subset \PP(\HH)$.
\end{lem}

\subsection{Standard modules}

Before turning to the main part of this work, let us comment on a side question. It is natural to wonder how restrictive the assumption $\PP\subset \Proj_{\CC}$ is. We will make two observations concerning that question. The first one is that one can slightly loosen the constraint by considering an inclusion of categories of partitions $\CC\subset \CC'$ and a standard module of projective partitions $\PP\subset \Proj_{\CC'}$. This still is a module of projective partitions over $\CC$, but all of its structure can be understood by looking at it as a standard module over $\CC'$. 
The second observation is that the previous weak form of standardness can be characterised in terms of $\PP$ only.

\begin{prop}
Let $\PP$ be a module of projective partitions over a category of partitions $\CC$. Then, there exists $\CC\subset \CC'$ such that $\PP\subset \Proj_{\CC'}$ is a module of projective partitions over $\CC'$ if and only if, for all $p, q\in \PP$, we have $qpq\in \PP$ as soon as the composition makes sense.
\end{prop}

\begin{proof}
If $\PP\subset \Proj_{\CC'}$ is a module of projective partitions over $\CC'$, then for any $p, q\in \PP$ we have $q\in \CC'$, hence $qpq\in \PP$ by definition. Conversely, let $\CC' = \langle \CC, \PP\rangle$ be the category of partitions generated by $\CC$ and $\PP$. By construction, $\PP\subset \Proj_{\CC'}$. Moreover, let $p\in \PP$ and $r\in \CC\cup \PP$. Then,
\begin{itemize}
\item If $r\in \CC$, then $rpr^{*}\in \PP$ if the composition makes sense, by definition of a module of projective partitions;
\item If $r\in \PP$, then $rpr^{*}\in \PP$ if the composition makes sense, by assumption.
\end{itemize}
As a consequence, for any $r\in \CC' = \langle \CC\cup \PP\rangle$, $rpr^{*}\in \PP$ if the composition makes sense, proving that $\PP$ is a module of projective partitions over $\CC'$.
\end{proof}

Let us call a module of projective partitions \emph{stable} if it satisfies the condition in the proposition above. We do not have an example of a non-stable module of projective partitions, and the question whether stability follows from the definition therefore remains open.

\section{Classification in the orthogonal case}\label{sec:orthogonal}

As an illustration of the results of Section \ref{sec:correspondence}, let us treat the case of orthogonal easy quantum groups. By this, we mean that we consider categories of partitions $\CC$ containing the partition $|$ with a white upper point and a black lower point. It is easy to see that using this, we can change the color of any partition in $\CC$ as we like, so that in the end we can forget about colors. The corresponding categories of partitions were completely classified in a series of works \cite{banica2009liberation}, \cite{banica2010classification}, \cite{weber2012classification}, \cite{raum2013easy}, \cite{raum2014combinatorics} and \cite{raum2013full}. If the partitions are furthermore assumed to be non-crossing, then there are exactly seven possibilities which were classified in \cite{banica2009liberation} and \cite{weber2012classification}.

\begin{thm}[Banica-Speicher, Weber]
The uncoloured categories of non-crossing partitions are the following ones
\begin{itemize}
\item $NC_{2}$ : all non-crossing partitions with all blocks of size $2$;
\item $NC_{1, 2}$ : all non-crossing partitions with all blocks of size at most $2$;
\item $NC_{1, 2}'$ : all partitions in $NC_{1, 2}$ with an even number of blocks of singletons;
\item $NC_{1, 2}^{\sharp}$ : all partitions in $NC_{1, 2}'$ with an even number of singletons between any two connected points;
\item $NC_{\text{even}}$ : all non-crossing partitions with all blocks of even size;
\item $NC'$ : all non-crossing partitions with an even number of blocks of odd size;
\item $NC$ : all non-crossing partitions.
\end{itemize}
\end{thm}

Before entering the details, let us introduce one additional example. We have already noticed that any category of partitions has at least two standard modules of projective partitions, namely $\Proj^{0}_{\CC}$ and $\Proj_{\CC}$. There is in fact a third one, namely
\begin{equation*}
\Proj_{\CC}^{2} = \langle \vert\odot\vert\rangle,
\end{equation*}
where the brackets denote the module of projective partitions generated by a set of projective partitions. Note however that it may in some cases coincide with $\Proj_{\CC}$. The corresponding dual quantum subgroup is called the \emph{projective version} of $\G_{N}(\CC)$ and denoted by $P\G_{N}(\CC)$. It is usually defined as the compact quantum group whose Hopf $*$-algebra is the sub-$*$-algebra of $\O(\G_{N}(\CC))$ generated by the coefficients of $u\otimes u$. Since $u\otimes u = u^{\vert\odot\vert}$, we see that indeed $\PP(P\G_{N}(\CC)) = \Proj^{2}_{\CC}$.

\subsection{Blocks of size at most two}

Let us start with the simplest case.

\begin{prop}\label{prop:classificationON+}
For $\CC = NC_{2}$, there are exactly three standard modules of projective partitions, namely $\Proj^{0}_{NC_{2}}$, $\Proj_{NC_{2}}^{2}$ and $\Proj_{NC_{2}}$.
\end{prop}

\begin{proof}
As shown in \cite[Ex 5.10]{freslon2014partition}, any projective partition in $NC_{2}$ is equivalent to $\vert^{\odot k}$ for some integer $k$. Let us consider the set $S(\PP) = \{k \mid \vert^{\odot k}\in \PP\}$. This set is stable under addition and letting the duality partition $\sqcup$ act, we see that if $k\in S$ and $k\geqslant 2$, then $k-2\in S$. As a consequence, there are only three prossibilities: $\{0\}$, $2\N$ and $\N$. Since by Proposition \ref{prop:characterizationprojective} a standard module of projective partitions is determined by the equivalence classes of its elements, hence by $S$, the proof is complete.
\end{proof}

Next, we study categories of partitions with blocks of size at most two. For $NC_{1, 2}$, it is known that the corresponding compact quantum group is isomorphic to $O_{N-1}^{+}$, hence by Theorem \ref{thm:correspondence} the result is the same as for $NC_{2}$ (and it can be proven by the same argument).

\begin{cor}\label{cor:classificationBN+}
For $\CC = NC_{1, 2}$, there are exactly three modules of projective partitions, namely $\Proj^{0}_{NC_{1,2}}$, $\Proj_{NC_{1,2}}^{2}$ and $\Proj_{NC_{1,2}}$.
\end{cor}

\begin{rem}\label{rem:OandB}
In both cases above, $\Proj^{0}_{\CC}$ will correspond to the trivial subgroup, while the two other ones correspond, as already mentioned, to the whole quantum group and to its projective version.
\end{rem}

There are two more examples of categories of non-crossing partitions with blocks of size at most two. Nevertheless, observing that they have the same projective partitions suggests that the classification will be the same for both, and it indeed is. The difference with the previous cases is that there are non-trough-block projective partitions which are not equivalent to one another. More precisely, if $s\in P(0, 1)$ denotes the singleton partition, then there are exactly two equivalence classes of non-through-block projective partitions, that of $\sqcap\sqcap^{*}$ and that of $ss^{*}$. This gives rise to an additional module of projective partitions.

\begin{prop}\label{prop:classificationBN+}
For $\CC = NC_{1, 2}^{\prime}$ and $\CC = NC_{1, 2}^{\sharp}$, there are exactly four standard modules of projective partitions, namely $\Proj^{0}_{\CC}$, $\langle \sqcap\sqcap^{*}\rangle$, $\Proj_{\CC}^{2}$ and $\Proj_{\CC}$.
\end{prop}

\begin{proof}
Let us first assume that $\PP$ contains a partition $p$ with a through-block, and observe that the only possible through-block in $\CC$ is $\vert$. If there is such a $p$ with an odd number of points in each row, then letting the duality partition $\sqcup$ act, we get that either $\vert\in \PP$ or $ss^{*}\odot\vert\odot ss^{*}\in \PP$. In the first case, we have $\PP = \Proj_{\CC}$ while in the second one, we can further reduce to obtain $ss^{*}\in \PP$. But then, $ss^{*}\odot ss^{*}\odot \vert\odot ss^{*}\odot ss^{*}\in \PP$ and we then conclude again that $\vert\in \PP$, hence $\PP = \Proj_{\CC}$.

We will therefore assume now that any partition in $\PP$ with a through-bock has an even number of points in each row, and we claim that the same must hold for all partitions in $\PP$. Indeed, if there is a partition with an odd number of points, then we can reduce it to $ss^{*}\in \PP$. But then, if $p$ has a through-block, so does $p\odot ss^{*}$ and the parity of the number of points has changed. This is a contradiction, hence the claim. Based on this, a partition $p\in \PP$ can be written as $p = p_{1}\odot \cdots \odot p_{2n}$ with $p_{i} = \vert$ or $p_{i} = ss^{*}$ for all $1\leqslant i\leqslant 2n$. Equivalently, $p$ is an horizontal concatenation of the following partitions: $\vert\odot\vert$, $ss^{*}\odot ss^{*}$, $\vert\odot ss^{*}$ and $ss^{*}\odot \vert$. The last one is conjugate to its predecessor, hence if we prove that the first three partitions are in $\Proj_{\CC}^{2}$, we will have $\PP = \Proj_{\CC}^{2}$. But $\vert\otimes\vert\in \Proj_{\CC}^{2}$ by definition and the other two can be obtained from it by letting $\vert\odot s$ act.

To conclude, we still have to consider the case $\PP\subset \Proj_{0}^{\CC}$. Then, either $\PP$ contains only partitions equivalent to $\sqcap\sqcap^{*}$, in which case it contains all of them by Proposition \ref{prop:characterizationprojective} and therefore $\PP = \langle \sqcap\sqcap^{*}\rangle$, or $\PP$ contains a partition equivalent to $ss^{*}$, in which case it contains up to equivalence all non-through-block projective partitions and the proof is complete.
\end{proof}

It is well known (see for instance \cite[Prop 3.2]{weber2012classification}) that the group of one-dimensional representations of $\G_{N}(\CC)$ is $\Z_{2}$ for $\CC = NC_{1, 2}^{\prime}$ and $\CC = NC_{1, 2}^{\sharp}$. This corresponds to $\Proj^{0}_{\CC}$ and contains the trivial subgroup, which corresponds to the module of all projective partitions which are equivalent to the empty partition, that is to say $\langle \sqcap\sqcap^{*}\rangle$.

\subsection{Blocks of arbitrary size}

If we now allow blocks of arbitrary size, it is natural to start with the category of all partitions with all blocks of even size, denoted by $NC_{\text{even}}$. Let $p_{4}\in P(2, 2)$ be the partition with one block and set $\Proj_{NC_{\text{even}}}^{1/2} = \langle p_{4}\rangle$.

\begin{prop}\label{prop:classificationHN+}
For $\CC = NC_{\text{even}}$, there are exactly four standard modules of projective partitions, namely $\Proj^{0}_{NC_{\text{even}}}$, $\Proj^{1/2}_{NC_{\text{even}}}$, $\Proj_{NC_{\text{even}}}^{2}$ and $\Proj_{NC_{\text{even}}}$.
\end{prop}

\begin{proof}
Let $\mathcal{Q}$ be the set of all projective partitions in $NC_{\text{even}}$ such that each through block has an even number of points on each row. This is a module of projective partitions and we claim that it equals $\Proj_{NC_{\text{even}}}^{1/2}$. One inclusion follows from the definition of $\mathcal{Q}$. As for the converse one, let $p\in \mathcal{Q}$. Up to equivalence, we can remove all its non-through-blocks so that $p = p_{1}\odot\cdots\odot p_{n}$ with $t(p_{i}) = 1$ and $p_{i}$ has an even number of points on each row. By \cite[Lem 5.12]{freslon2013representation}, each $p_{i}$ is equivalent to $p_{4}$, hence $p\in \Proj_{NC_{\text{even}}}^{1/2}$.

Let us now consider a standard module of projective partitions $\PP$ over $NC_{\text{even}}$ which is not contained in $\Proj_{NC_{\text{even}}}^{0}$. Any projective partition in it is equivalent to one of the form $p_{1}\odot\cdots\odot p_{n}$ with $p_{i}\in \{\vert, p_{4}\}$ for all $1\leqslant i\leqslant n$. If there is such a partition with an odd number of points on each row, then letting $\sqcup$ act we get $\vert\in \PP$, hence $\PP = \Proj_{NC_{\text{even}}}$. Let us therefore assume that the number of points in each row is even. Then, if $p_{i} = p_{4}$ for all $1\leqslant i\leqslant n$ and all partitions $p$, we have $\PP\subset \Proj_{NC_{\text{even}}}^{1/2}$. Since we can also produce $p_{4}$ from $p$ by letting $\sqcup$ act, we conclude that $\PP = \Proj_{NC_{\text{even}}}^{1/2}$. Otherwise, using $\sqcup$ again we can reduce the partition to $\vert\odot \vert$, so that $\PP = \Proj^{2}_{NC_{\text{even}}}$.

If $\PP\subset \Proj_{NC_{\text{even}}}^{0}$, then since any non-through-block projective partition is equivalent to the empty one, we have equality and the proof is complete.
\end{proof}

\begin{rem}\label{rem:classificationHN+}
The additional module of projective partitions $\Proj_{NC_{\text{even}}}^{1/2}$ corresponds to the dual quantum subgroup whose Hopf $*$-algebra is generated by the coefficients of the representation $u^{p_{4}}$. An analysis of the fusion rules shows that this quantum group is in fact isomorphic to $S_{N}^{+} = \G_{N}(NC)$.
\end{rem}

The last two cases can be treated as the previous ones.

\begin{prop}\label{prop:classificationSN+}
For $\CC = NC^{\prime}$, there are exactly three modules of projective partitions, namely $\Proj_{NC^{\prime}}^{0}$, $\langle \sqcap\sqcap^{*}\rangle$, $\Proj_{NC^{\prime}}$. As for $\CC = NC$, there are exactly two modules of projective partitions, namely $\Proj_{NC}^{0}$ and $\Proj_{NC}$.
\end{prop}

\begin{proof}
The first case is done as in the proof of Proposition \ref{prop:classificationBN+}. As for the second one, note that using the partition $p_{3}\in NC(2, 1)$ with one block, we can turn a projective partition with an even number of through-blocks into a partition with an odd number of through blocks. Since any projective partition is equivalent in $NC$ to $\vert^{\odot n}$ for some integer $n$ by \cite[Ex 5.9]{freslon2013representation}, we conclude that $\Proj_{NC}^{2} = \Proj_{NC}$ and the proof is completed as that of Proposition \ref{prop:classificationON+}.
\end{proof}

If crossings are allowed, then the proof of Theorem \ref{thm:correspondence} breaks down. More precisely, standard modules of projective partitions still yield dual quantum subgroups of $\G_{N}(\CC)$, but the converse does not hold any more. For instance, projective partition in the category $P_{2}$ of all partitions with blocks of size $2$ are classified up to equivalence by their number of through-blocks, so that the same proof as in Proposition \ref{prop:classificationON+} shows that there are exactly three standard modules of projective partitions over $P_{2}$. Nevertheless, dual quantum subgroups of a classical compact group correspond to quotients, and the orthogonal group $O_{N}$ has four quotients. The one which does not come from a module of projective partitions is $\Z_{2}$, which is the quotient by the normal subgroup $SO_{N}$.

We summarize the results of this section in the following table referring to \cite{weber2012classification} for the notation of the quantum groups.\\
\begin{center}
\begin{tabular}{l|l|l|l}\label{table:overvieworthogonalmodules}
Category of partitions & Modules of projective partitions & Quantum groups & References \\\hline\hline
$NC_2$ and $NC_{1,2}$
&$\Proj_{NC_{2}}$ and $\Proj_{NC_{1,2}}$
&$\widehat{O_N^+}$ and $\widehat{B_N^+}$
&Prop. \ref{prop:classificationON+}, \\
&$\Proj_{NC_{2}}^{2}$ and $\Proj_{NC_{1,2}}^{2}$&$\widehat{PO_N^+}$ and $\widehat{PB_N^+}$
&Cor. \ref{cor:classificationBN+},\\
&$\Proj_{NC_{2}}^{0}$ and  $\Proj_{NC_{1,2}}^{0}$  &trivial group
&Rem. \ref{rem:OandB}\\\hline
$NC_{1, 2}^{\prime}$ and $NC_{1, 2}^{\sharp}$
&$\Proj_{NC_{1, 2}^{\prime}}$ and $\Proj_{NC_{1, 2}^{\sharp}}$
&$\widehat{{B_N'}^+}$ and $\widehat{B_N^{\sharp +}}$
&Prop. \ref{prop:classificationBN+}\\
&$\Proj_{NC_{1, 2}^{\prime}}^{2}$ and $\Proj_{NC_{1, 2}^{\sharp}}^{2}$
&$\widehat{P{B_N'}^+}$ and $\widehat{P{B_N}^{\sharp+}}$\\ 
&$\Proj_{NC_{1, 2}^{\prime}}^{0}$ and $\Proj_{NC_{1, 2}^{\sharp}}^{0}$
&$\Z_2$\\ 
&$\langle \sqcap\sqcap^{*}\rangle$
&trivial group\\\hline
$NC_{\text{even}}$
&$\Proj_{NC_{\text{even}}}$
&$\widehat{{H_N}^+}$
&Prop. \ref{prop:classificationHN+},\\
&$\Proj_{NC_{\text{even}}}^{2}$
&$\widehat{P{H_N}^+}$
&Rem. \ref{rem:classificationHN+} \\
&$\Proj^{1/2}_{NC_{\text{even}}}$
&$\widehat{S_N^+}$\\
&$\Proj^{0}_{NC_{\text{even}}}$
&trivial group\\\hline
$NC$ 
&$\Proj_{NC}$
&$\widehat{{S_N}^+}$
&Prop. \ref{prop:classificationSN+}\\
&$\Proj_{NC}^{0}$
&trivial group\\\hline
$NC^{\prime}$
&$\Proj_{NC^{\prime}}$
&$\widehat{{S_N'}^+}$
&Prop. \ref{prop:classificationSN+}\\
&$\Proj_{NC^{\prime}}^{0}$
&$\Z_2$\\
&$\langle \sqcap\sqcap^{*}\rangle$
&trivial group
\end{tabular}
\end{center}
\quad\\

\section{Classification of modules in the unitary case}\label{sec:unitary}

We will now apply our tools to the study of the free unitary quantum groups $U_{N}^{+}$. So far, nothing is known concerning the dual quantum subgroups of $U_{N}^{+}$, except for the projective versions (see below). As we will see, the combinatorics of modules of projective partitions is in that case quite elementary and enables a full classification. Once that classification is done, we will study in more detail the corresponding compact quantum groups.

We will first classify the standard modules of projective partitions over the category of partitions $\CC_{U}$ of $U_{N}^{+}$. To do this, let us recall some basic facts concerning this object (see for instance \cite[Thm 4.16]{tarrago2015unitary} for a proof).

\begin{prop}\label{prop:defcategoryUN+}
Let $\CC_{U}$ be the set of all pair partitions such that when two points are connected, they have the same color if they are in different rows and different colors if they are in the same row. Then, $\CC_{U}$ is a category of partitions, and the corresponding compact quantum group is $U_{N}^{+}$.
\end{prop}

The classification will be done in two steps: first, we will define three one-parameter families of modules of projective partitions, and then we will use them to describe all the other modules. For the sake of clarity, we will use a number of notations and conventions in the sequel, which we now detail.

Given a word $w$ on $\{\circ, \bullet\}$, its \emph{color balance} $\cb(w)$ is the difference between the number of white points and the number of black points. For a word $w = w_{1}\cdots w_{n}$, we will write $w_{\geqslant l} = w_{l}\cdots w_{n}$, and define similarly $w_{>l}$, $w_{\leqslant l}$ and $w_{<l}$. Given a word $w = w_{1}\cdots w_{n}$, we define the following projective partition:
\begin{equation*}
p_{w} = \id_{w_{1}}\odot\cdots\odot\id_{w_{n}}.
\end{equation*}

We start with an equivalent description of standard modules of projective partitions in terms of very elementary combinatorial objects. To do this, we will need some further shorthand notations. If $w = w_{1}\cdots w_{n}$ is a word on $\{\circ, \bullet\}$, we set $\overline{\circ} = \bullet$, $\overline{\bullet} = \circ$ and define the \emph{conjugate} of $w$ to be
\begin{equation*}
\overline{w} = \overline{w}_{n}\cdots \overline{w}_{1}.
\end{equation*}
Moreover, given two words $w$ and $w'$, we will denote by $w.w'$ their concatenation.

\begin{de}
A set of words on $\{\circ, \bullet\}$ is said to be \emph{admissible} if it is invariant under the following operations:
\begin{itemize}
\item Concatenation of words;
\item Conjugation;
\item Cancellation of a subword of the form $\circ\bullet$ or $\bullet\circ$.
\end{itemize}
The cancellations above will be called \emph{elementary} in the sequel.
\end{de}

\begin{lem}
Standard modules of projective partitions over $\CC_{U}$ are in one-to-one correspondance with admissible sets of words.
\end{lem}

\begin{proof}
Let $\PP$ be a standard module of partitions and let $p\in \PP$. Then, as explained in \cite[Sec 6.3]{freslon2013representation}, there exists a unique word $w$ such that $p\sim p_{w}$. The set $W(\PP)$ of all such words is stable under concatenation since $p_{w.w'} = p_{w}\odot p_{w'}$, as well as under conjugation since $\overline{p_{w}} = p_{\overline{w}}$. Moreover, letting $D_{\circ\bullet}$ or $D_{\bullet\circ}$ act shows that $W(\PP)$ is also stable under elementary cancellations, hence $W(\PP)$ is admissible.

Let now $W$ be an admissible set of words and let $\PP(W)$ be the set of all projective partitions in $\CC_{U}$ which are equivalent to $p_{w}$ for some $w\in W$. Because $W$ is stable under concatenation, $\PP(W)$ is stable under $\odot$ and because $W$ is stable under conjugation, $\overline{\PP(W)} = \PP(W)$. Moreover, if $p\in \PP(W)$ and $q\prec p$, then there exists $w\in W$ such that $q$ is equivalent to a projective partition $q'\prec p_{w}$. But $q'$ can then be written as $b_{0}\odot\id_{v_{1}}\odot b_{1}\odot\cdots\odot\id_{v_{l}}\odot b_{l}$ where $v_{1}\cdots v_{l}$ is a subword of $w$ and $b_{i}$ is projective with $t(b_{i}) = 0$ for all $0\leqslant i\leqslant l$. Now, by Proposition \ref{prop:defcategoryUN+}, each $b_{i}$ is a nesting of duality partitions, so that $v$ is obtained from $w$ by applying the corresponding cancellations, which are elementary, and the proof is complete. 
\end{proof}

\subsubsection{The three series}

We now introduce the main objects of this section, namely three infinite series of admissible sets of words (equivalently, standard modules of projective partitions over $\CC_{U}$). A natural idea to produce such sets is to take a word and try to describe the smallest admissible set that contains it, which will be said to be \emph{generated} by it. For convenience, we will however first define abstract admissible sets, and then give explicit generators. Let us start with the simplest case:

\begin{de}
For $k > 0$, we set $W^{(k)} = \{w \mid \cb(w) = 0 \mod k\}$.
\end{de}

By definition, $\circ^{k}\in W^{(k)}$ hence one may hope for an equality of the corresponding admissible sets of words. That this is indeed the case is the subject of the next result.

\begin{prop}\label{prop:wk}
The set $W^{(k)}$ is admissible and generated by $\circ^{k}$.
\end{prop}

\begin{proof}
It is clear that $W^{(k)}$ is stable under concatenation and conjugation, and since elementary cancellations remove to different letters, they do not change the color balance, hence $W^{(k)}$ is admissible.

Obviously, $\circ^{k}\in W^{(k)}$, hence $W = \langle \circ^{k}\rangle\subset W^{(k)}$. Conversely, let $w\in W^{(k)}$. It can be written as
\begin{equation*}
w = w_{1}^{i_{1}}\overline{w}_{1}^{i_{2}}\cdots w_{1}^{i_{2l-1}}\overline{w}_{1}^{i_{2l}}.
\end{equation*}
Consider the euclidean division $i_{1} = d_{1}k + j_{1}$ and set, for $t\geqslant 1$, $i'_{t+1} = i_{t+1} + k - j_{t}$ and $i'_{t+1} = d_{t+1}k - j_{t+1}$ with $0\leqslant j_{t+1} < k$. Then, $w$ can be obtained through elementary cancellations from the word
\begin{equation*}
w_{1}^{i_{1}}w_{1}^{k-j_{1}}\overline{w}_{1}^{k-j_{1}}\overline{w}_{1}^{i_{2}}\overline{w}_{1}^{k-j_{2}}w_{1}^{k-j_{2}}\:\cdots\: w_{1}^{i_{2l-1}}w_{1}^{k-j_{2l-1}}\overline{w}_{1}^{k-j_{2l-1}}\overline{w}_{1}^{i_{2l}} = w_{1}^{i_{1} + k - j_{1}}\overline{w}_{1}^{i'_{2} + k - j_{2}}\cdots w_{1}^{i'_{2l-1} + k - j_{2l-1}}\overline{w}_{1}^{i_{2l} + k - j_{2l-1}}.
\end{equation*}
This is a concatenation of elements of $W$, except perhaps for the last term $\overline{w}_{1}^{i_{2l} + k - j_{2l-1}}$. We therefore have to prove that $i_{2l}' = i_{2l} + k - j_{2l-1} = 0 \mod k$.

To show this, we first claim that we have, for all $1\leqslant t\leqslant l$,
\begin{align*}
i_{2t}' & = \sum_{s=1}^{2t}(-1)^{s}i_{s} \mod k \\
i_{2t-1}' & = -\sum_{s=1}^{2t-1}(-1)^{s}i_{s} \mod k.
\end{align*}
Let us prove this by induction.
\begin{itemize}
\item If $t = 1$, then
\begin{equation*}
i_{2}' = i_{2} + k - j_{1} = i_{2} + k + (d_{1}k - i_{1}) = (d_{1}+1)k + i_{2} - i_{1} = i_{2} - i_{1} \mod k.
\end{equation*}
\item Assume that the result holds for some $t\geqslant 1$. Then
\begin{equation*}
i_{t+1}' = i_{t+1} + k - j_{t} = i_{t+1} - i'_{t} \mod k 
\end{equation*}
and the claim follows. 
\end{itemize}
The claim now yields
\begin{equation*}
i_{2l}' = \sum_{t=1}^{k} i_{2t} - i_{2t-1} \mod k.
\end{equation*}
The sum above is nothing but the opposite of the color balance $\cb(w)$, which is a multiple of $k$ by assumption. The proof is therefore complete.
\end{proof}

This gives us our first infinite series of admissible sets of words, and we note that, by definition, $W^{(k)}\subset W^{(k')}$ if and only if $k'\mid k$. The two other series are somehow dual to each other, and are given by generators with vanishing color balance.

\begin{de}
We define, for $0 < k \leqslant \infty$,
\begin{align*}
W^{(\circ, k)} & = \{w = w_{1}\cdots w_{n} \mid \cb(w) = 0 \text{ and } k\geqslant \cb(w_{\leqslant l})\geqslant 0, \:\forall 1\leqslant l\leqslant n\} \\
W^{(\bullet, k)} & = \{w = w_{1}\cdots w_{n} \mid \cb(w) = 0 \text{ and } -k\leqslant \cb(w_{\leqslant l})\leqslant 0, \:\forall 1\leqslant l\leqslant n\}.
\end{align*}
\end{de}

\begin{prop}
The sets $W^{(\circ, k)}$ and $W^{(\bullet, k)}$ are admissible and generated respectively by $\circ^{k}\bullet^{k}$ and by $\bullet^{k}\circ^{k}$.
\end{prop}

\begin{proof}
We will only prove the statement for $W^{(\circ, k)}$, the proof for $W^{(\bullet, k)}$ being similar. It is clear that the set is stable under concatenation and elementary cancellations since the latter do not change the color balance. As for conjugation, observe that
\begin{equation*}
c_{\leqslant l}(\overline{w}) = -c_{>l}(w) = -c(w) + c_{\leqslant l}(w) = c_{\leqslant l}(w).
\end{equation*}
Therefore, $W^{(\circ, k)}$ is admissible.

By definition, $\circ^{k}\bullet^{k}\in W^{(\circ, k)}$, hence $W = \langle \circ^{k}\bullet^{k}\rangle\subset W^{(\circ, k)}$. Conversely, given $w\in W^{(\circ, k)}$, it can be written as
\begin{equation*}
w = \circ^{i_{1}}\bullet^{i_{2}}\cdots \circ^{i_{2l-1}}\bullet^{i_{2l}}
\end{equation*}
with $i_{t}\leqslant k$ for all $1\leqslant t\leqslant 2l$. Let us set
\begin{equation*}
i_{t}' = \sum_{s=1}^{t}(-1)^{s+1}i_{t}.
\end{equation*}
Since $i_{t}' = \cb(w_{\leqslant t})$, this is positive, hence it makes sense to consider the following word:
\begin{equation*}
v = \circ^{i_{1}}\bullet^{i_{2}}\bullet^{i_{1} - i_{2}}\circ^{i_{1}-i_{2}}\circ^{i_{3}}\bullet^{i_{4}}\bullet^{i_{1} - i_{2} + i_{3} - i_{4}}\circ^{i_{1} - i_{2} + i_{3} - i_{4}}\cdots = \circ^{i_{1}}\bullet^{i_{1}}\circ^{i'_{2}}\bullet^{i_{2}'}\cdots.
\end{equation*}
Clearly, $w$ can be obtained from $v$ by elementary cancellations. Therefore, we only have to prove that $v\in W$, and since $v$ is the concatenation of the words $\circ^{i'_{2t-}}\bullet^{i'_{2t-}}$, it is enough to show that $i'_{2t-1}\leqslant k$. But as we already noticed, $i'_{2t-1} = \cb(w_{\leqslant 2t - 1})$, hence the proof is complete.
\end{proof}

\begin{rem}
Observe that $W^{(\circ, k)}\subsetneq W^{(\circ, k')}$, so that these yield an infinite ascending chain of dual quantum subgroups. This recovers the result of \cite[Thm 6.14]{cirio2014connected} stating that the fusion ring of $U_{N}^{+}$ is not Noetherian, contrary to fusion rings of classical compact Lie groups.
\end{rem}

\subsubsection{The classification}

We will now give the complete list of sets of admissible words. For the sake of clarity, we first give a separate lemma.

\begin{lem}\label{lem:simplificationword}
Let $w\in W^{(\circ, k)}$ and assume that there exist $l$ such that $c(w_{\leqslant l}) = k$. Then, using elementary cancellations, $w$ can be reduced to the word $\circ^{k}\bullet^{k}$.
\end{lem}

\begin{proof}
We will write $w = \circ^{i_{1}}\bullet^{j_{1}}\cdots\circ^{i_{n}}\bullet^{j_{n}}$ and proceed by induction on $n$. Il $n = 1$, then by assumption $w = \circ^{k}\bullet^{k}$. Let us therefore assume that the result holds for some $n$ and consider $w = \circ^{i_{1}}\bullet^{j_{1}}\cdots\circ^{i_{n+1}}\bullet^{j_{n+1}}$. Let $l$ be such that $c(w_{\leqslant l}) = k$, and observe that $l = i_{1} + j_{1} + \cdots + i_{t}$ for some $1\leqslant t\leqslant n+1$. Let us assume for the moment that $t\leqslant n$ and observe that
\begin{equation*}
i_{n+1} - j_{n+1} = - c(w_{\leqslant i_{1} + j_{1} + \cdots + i_{n} + j_{n}}) \leqslant 0
\end{equation*}
so that $j_{n+1}\geqslant i_{n+1}$ and we can use elementary cancellations to reduce $w$ to
\begin{equation*}
\widetilde{w} = \circ^{i_{1}}\bullet^{j_{1}}\cdots \circ^{i_{n}}\bullet^{j_{n} + j_{n+1} - i_{n+1}}.
\end{equation*}
We still have $\widetilde{w}\in W^{(\circ, k)}$ and moreover, $c(\widetilde{w}_{\leqslant l}) = c(w_{\leqslant l}) = k$, hence we can conclude by induction.

If instead $t = n+1$, then because $i_{1} - j_{1} = c(w_{\leqslant i_{1} + j_{1}}) \geqslant 0$ we can use elementary cancellations to reduce $w$ to
\begin{equation*}
w' = \circ^{i_{1} - j_{1} + i_{2}}\bullet^{j_{2}}\cdots \circ^{i_{n+1}}\bullet^{j_{n+1}}.
\end{equation*}
Again, $w'\in W^{(\circ, k)}$ and since we have removed a subwords with vanishing color balance,
\begin{equation*}
c\left(w'_{\leqslant i_{1} - j_{1} + i_{2} + j_{2} + \cdots + i_{n+1}}\right) = c(w_{\leqslant l}) = k,
\end{equation*}
so that we can conclude once more by induction.
\end{proof}

With this, we are in position to state and prove our classification theorem.

\begin{thm}\label{thm:classificationunitary}
The admissible sets of words are exactly the following ones:
\begin{itemize}
\item The empty set;
\item $W^{(k)}$, $W^{(\circ, k)}$, $W^{(\bullet, k)}$ for $k\in \N$;
\item $W^{(k, k')} = \left\langle W^{(\circ, k)}, W^{(\bullet, k')}\right\rangle$ for $k, k'\in \N$;
\end{itemize}
\end{thm}

\begin{proof}
Let $W$ be a non-empty admissible set of words and let us first assume that $W$ is generated by one word $w$. If $\cb(w) \neq 0$, then by using elementary cancellations we have that $\circ^{\cb(w)}\in W$ or $\bullet^{\cb(w)}\in W$. The two cases are similar, hence we will only consider the first one. By Proposition \ref{prop:wk} $W^{(\cb(w))}\subset W$. But $w\in W^{\cb(w)}$ by definition, so that $W = W^{\cb(w)}$. If $\cb(w) = 0$, let us assume that the first letter is $\circ$ (the other case being similar) and set
\begin{equation*}
k = \max\{\cb(w_{\leqslant l}) \mid 1\leqslant l\leqslant \vert w\vert\},
\end{equation*}
so that $W \subset W^{(\circ, k)}$. Moreover, by Lemma \ref{lem:simplificationword}, we can reduce $w$ to $\circ^{k}\bullet^{k}$. In other words, $W^{(\circ, k)} = \langle \circ^{k}\bullet^{k}\rangle\subset W$.

If now $W$ is arbitrary, then it is generated by the admissible sets of words generated by each of its elements. The result therefore follows from the following elementary observations:
\begin{itemize}
\item $\left\langle W^{(k)}, W^{(k')}\right\rangle = W^{\gcd(k, k')}$. Indeed, the left-hand side is by definition contained in the right-hand side. Moreover, if $ak - bk' = \gcd(k, k')$ with $a, b\in \N$, then $\circ^{ak}\bullet^{bk'}\in W^{(k)}\cup W^{(k')}$ and generates $W^{\gcd(k, k')}$. If instead $a$ and $b$ are negative, then considering $\circ^{bk'}\bullet^{ak}$ yields the same result.
\item By mere definition, $W^{(\circ, k)}, W^{(\bullet, k)}\subset W^{(k')}$ for all $k, k'\in \N$.
\item Still by definition, $W^{(\circ, k)}\subset W^{(\circ, k')}$ and $W^{(\bullet, k)}\subset W^{(\bullet, k')}$ for all $k\leqslant k'$.
\end{itemize}
\end{proof}

Let us point out two interesting facts appearing in Theorem \ref{thm:classificationunitary}. The first one is that there are three admissible sets of words which are not finitely generated. This means that the corresponding compact quantum groups are not compact matrix quantum groups. Thinking of $U_{N}^{+}$ as the dual of a free quantum group, this is not so surprising since it is well known that a finitely generated free group has subgroups which are not finitely generated. The second fact is that besides these three exceptions, all dual quantum subgroups of $U_{N}^{+}$ are generated by one or two irreducible representations. Besides, as we will see below, when two generators are needed then there is a simple way of describing the quantum group as a free product where each factor involves only one generator.

\section{Applications tp quantum subgroups of $\widehat{U_N^+}$} \label{sec:quantumsubgroups}

We now study the quantum groups corresponding to the modules arising in the previous section.

\subsection{Fusion semi-rings}\label{subsec:fusionrings}

In view of Theorem \ref{thm:classificationunitary}, it is natural  to investigate the compact quantum groups appearing as dual quantum subgroups of $U_{N}^{+}$. Let us denote by $\G_{N}(W)$ the dual quantum subgroup of $U_{N}^{+}$ corresponding to a set of words $W$. Since these are defined through their representation category, we may first try to compute the corresponding fusion semi-rings.

Before embarking into this, let us make a few remarks to simplify our work:
\begin{itemize}
\item The sets $W^{(\circ, k)}$ and $W^{(\bullet, k)}$ yield anti-isomorphic quantum groups (see Proposition \ref{prop:antiisomorphic}), hence anti-isomorphic fusion rings. We will therefore only consider the first one;
\item The quantum group corresponding to $W^{(k, k')}$ is the free product of those corresponding to $W^{(\circ, k)}$ and $W^{(\bullet, k')}$ (see Proposition \ref{prop:antiisomorphic}). This means that the fusion ring of the former can be recovered from those of the latter by \cite[Thm 3.10]{wang1995free}. We will therefore not consider $W^{(k, k')}$.
\end{itemize}

Let us now recall a few facts concerning fusion semi-rings. Given a compact quantum group $\G$, we denote by $R^{+}(\G)$ the free abelian semi-group on the set $\Irr(\G)$ of equivalence classes of irreducible representations of $\G$. It can be endowed with a multiplicative structure, usually denoted by $\otimes$, through the formula
\begin{equation*}
\alpha\otimes \beta = \sum_{\gamma\subset \alpha\otimes \beta}m(\gamma, \alpha\otimes \beta)\gamma,
\end{equation*}
where $m(\gamma, \alpha\otimes\beta)$ denotes the multiplicity of the inclusion. It is also endowed with an antilinear and anti-multiplicative involution sending $\alpha$ to its conjugate representation $\overline{\alpha}$. The data of the semi-ring called $R^{+}(\G)$ together with its involution is called the \emph{fusion semi-ring} of $\G$.

We will now compute the fusion rings corresponding to sets of words generated by a symetric element, that is to say $W^{(\circ, k)}$. We already have some information: for $k = 1$, this is by definition the projective version of $U_{N}^{+}$, which is known to be isomorphic to the quantum automorphism group of $M_{N}(\C)$ equipped with the trace (see for instance \cite[Prop 3.1]{banica2010invariants}). Because that quantum group will be extensively used in the sequel, we will denote it by $\QA$.

Our final result will be that $\G(W^{(\circ, k+1)})$ can be obtained from $\G(W^{(\circ, k)})$ using a construction called the \emph{free wreath product}, but for the moment we just want to prove a version of this at the level of fusion semi-rings. We will therefore use an ad hoc definition.

\begin{de}
Let $\G$ be a compact quantum group and let $R^{+, \wr}(\G)$ be the free abelian semi-group on all words on $\Irr(\G)$. Given two words $w = \alpha_{1}\cdots \alpha_{n}$ and $w' = \beta_{1}\cdots \beta_{k}$ on $\Irr(\G)$, we set
\begin{align*}
w\otimes w' & = \alpha_{1}\cdots \alpha_{n}\beta_{1}\cdots \beta_{k} + \sum_{\gamma\subset \alpha_{n}\otimes \beta_{1}}m(\gamma, \alpha_{n}\otimes \beta_{1})\alpha_{1}\cdots \alpha_{n-1}\gamma \beta_{1}\cdots \beta_{k} \\
& + \delta_{\alpha_{n} = \overline{\beta}_{1}} (\alpha_{1}\cdots \alpha_{n-1})\otimes (\beta_{1}\cdots \beta_{k}).
\end{align*}
This inductively defines a semi-ring structure on $R^{+, \wr}(\G)$. We moreover define an involution by $\overline{w} = \overline{\alpha}_{n}\cdots \overline{\alpha}_{1}$.
\end{de}

Note that $R^{+, \wr}$ is generated, as a semi-ring, by the elements of $\Irr(\G)$ seen as words with one letter. Let us introduce some additional shortand notations for convenience. The fusion semi-ring of $\G_{N}(W)$ will simply be denoted by $R^{+}(W)$. Moreover, irreducible representations of the former quantum group will be denoted by the corresponding word in $W$ so that $R^{+}(W)$ is freely generated as an abelian group by the elements of $W$. We can now describe $R^{+}(W^{(\circ, k+1)})$ in terms of $R^{+, \wr}(W^{(\circ, k)})$.

\begin{prop}\label{prop:fusionsymetric}
Let $k\in \N$. Set, for $w\in W^{(\circ, k)}$, $a(w) = \circ w\bullet$. Then exists a semi-ring isomorphism
\begin{equation*}
\Psi : R^{+, \wr}\left(W^{(\circ, k)}\right) \to R^{+}\left(W^{(\circ, k+1)}\right)
\end{equation*}
such that $\Phi(w) = a(w)$ for all $w\in W^{(\circ, k)}$.
\end{prop}

\begin{proof}
We first claim that any element of $W^{(\circ, k+1)}$ can be written as a word on elements of the form $a(w')$ for $w'\in W^{(\circ, k)}$. Let us prove this by induction. The smallest word in $W^{(\circ, k+1)}$ is $\circ \bullet$, which is $a(\emptyset)$. Assuming that the result holds for all words of length at most $n$ for some $n\geqslant 2$, let us consider $w = w_{1}\cdots w_{n+2}$ and set 
\begin{equation*}
i = \min\{j \mid \cb(w_{\leqslant j}) = 0\}.
\end{equation*}
Then, by definition of $W^{(\circ, k+1)}$, we have $w_{1} = \circ$, $w_{i} = \bullet$ and
\begin{equation*}
\cb(w_{i+1}\cdots w_{j}) = \cb(w_{\leqslant j})
\end{equation*}
for all $j\geqslant i+1$, so that $w_{>i}\in W^{(\circ, k+1)}$. Moreover, by definition of $i$, for any $2\leqslant j\leqslant i-1$,
\begin{equation*}
k\geqslant \cb(w_{2}\cdots w_{j}) = \cb(w_{\leqslant j}) - 1 \geqslant 0
\end{equation*}
Summing up, $w' = w_{2}\cdots w_{i-1}\in W^{(\circ, k)}$ and $w = a(w')w_{>i}$. The claim then follows by induction.

Let us now set $\Psi(w) = a(w)$ for $w\in W^{(\circ, k)}$ and extend it to words on $W^{(\circ, k)}$ by simply concatenating the images. The claim above shows that $\Psi : R^{+, \wr}(W^{(\circ, k)}) \to R^{+}(W^{(\circ, k)})$ is a surjective morphism of abelian groups and we will now prove that it is also injective. For that purpose, let $v\in W^{\circ, k+1}$ and write it as $v = a(w^{1})\cdots a(w^{n})$ with a ``maximal'' decomposition in the sense that $\cb(a(w^{i})_{\leqslant j}) > 0$ for all $j < \vert w^{i}\vert + 2$ and $1\leqslant i\leqslant n$ (this is exactly the decomposition provided by the procedure above). Let us assume that there is another decomposition $v = a(w^{\prime 1})\cdots a(w^{\prime m})$ and observe that because of the maximality condition, there exists $n_{1}$ such that $a(w^{\prime 1}) = a(w^{1})\cdots a(w^{n_{1}})$. If $n_{1} > 1$, then
\begin{equation*}
\cb\left(w^{\prime 1}_{\leqslant \vert w_{1}\vert + 1}\right) = \cb(\circ w_{1}\bullet) - \cb(\circ) = -1,
\end{equation*}
contradicting $w^{\prime 1}\in W^{\circ, k}$. Thus, $w^{\prime 1} = w^{1}$ and a straightforward induction shows that $m = n$ and $w^{\prime i} = w^{i}$ for all $1\leqslant i\leqslant n$, proving injectivity.

To conclude, we must now show that $\Psi$ also respects the mutliplicative structure and the involution. The second one is clear from the fact that $\overline{a(w)} = a(\overline{w})$. As for the first one, it is enough to check it on the generators and this is a simple calculation:
\begin{align*}
a(w)\otimes a(w') & = (\circ w\bullet).(\circ w'\bullet) \\
& = \circ w\bullet\circ w\bullet + (\circ w).(w'\bullet) \\
& = a(w)a(w') + \circ(w\otimes w')\bullet \\
& = a(w)a(w') + \circ\left(\sum_{w = az, w' = \overline{z}b}ab\right)\bullet \\
& = a(w)a(w') + \sum_{w = az, w' = \overline{z}b}a(ab)
\end{align*}
\end{proof}

Using this, we will prove in Theorem \ref{thm:iteratedwreathsymetric} that $\G_{N}(W^{(\circ, k)}))$ can be obtained by iterated free wreath product by $\QA$.

\subsection{Dual quantum subgroups}

We will now complete our classification by giving an explicit description of all the dual quantum subgroups of $U_{N}^{+}$. The computation of fusion rings already gives natural candidates, and we will see how we can use the results of Section \ref{subsec:fusionrings} to compute the corresponding dual quantum subgroups. 

\subsubsection{The symetric case}

But first things first, we will start by proving the claims at the beginning of Section \ref{subsec:fusionrings}.

\begin{prop}\label{prop:antiisomorphic}
For any integer $k$, there is an anti-isomorphism of CQG-algebras
\begin{equation*}
G\left(W^{(\circ, k)}\right)\cong \G\left(W^{(\bullet, k)}\right).
\end{equation*}
Moreover, there is a CQG-algebra isomorphism
\begin{equation*}
\G\left(W^{(k, k')}\right) \cong \G\left(W^{(\circ, k)}\right)\ast\G\left(W^{(\bullet, k')}\right)
\end{equation*}
\end{prop}

\begin{proof}
Let $S$ be the antipode of $U_{N}^{+}$, which is the unique $*$-anti-homomorphism such that $S(u_{ij}) = u_{ji}^{*}$. Then, applying $S$ coefficient-wise to representations, we have the equality $S(u^{\circ\otimes k}\otimes u^{\bullet\otimes k}) = u^{\bullet\otimes k}\otimes u^{\circ\otimes k}$ and it follows from this by a straightforward induction that
\begin{equation*}
S(u^{\circ^{k}\bullet^{k}}) = u^{\bullet^{k}\circ^{k}}.
\end{equation*}
As a consequence $S(\O(G_{N}(W^{(\circ, k)}))) = \O(G_{N}(W^{(\bullet, k)}))$ and the result follows from the fact that $S$ is an anti-isomorphism.

As for the second statement, it is enough to prove that if $u^{1}, \cdots, u^{p}$ are non-trivial irreducible representations of $\G(W^{(\circ, k)})$ and $v^{1}, \cdots, v^{p}$ are non-trivial irreducible representations of $\G(W^{(\bullet, k)})$, then $u^{1}\otimes v^{1}\otimes \dots\otimes u^{p}\otimes v^{p}$ is irreducible. To do this, consider the words $w$ and $w'$ corresponding to the representations $u^{1}$ and $v^{1}$ respectively. By definition, $w$ ends with $\bullet$ and $w'$ starts with $\bullet$. As a consequence,
\begin{equation*}
u^{1}\otimes v^{1} = u^{w}\otimes u^{w'} = u^{ww'}
\end{equation*}
is irreducible. One can then work by induction to deduce the result.
\end{proof}

We will now start by describing $\G_{N}(W^{(\circ, k)})$. The description will involve the use of the quantum automorphism group of $M_{N}(\C)$ with respect to its canonical trace, which we denote by $\QA$, see \cite{wang1998quantum, banica1999symmetries} for a definition. It will also involve the free wreath product with respect to that quantum group, the definition of which we recall (see \cite{fima2015free} for details and \cite{bichon2004free} for the original definition).

\begin{de}\label{de:freewreath}
Let $\G$ be a compact quantum group and let $A(\G)$ be the universal $*$-algebra generated by elements $a(\alpha)_{ij}$ where $\alpha\in \Irr(\G)$ and $1\leqslant i, j\leqslant N^{2}\dim(\alpha)^{2}$ such that
\begin{itemize}
\item Each matrix $a(\alpha)\in \mathcal{L}(M_{N}(\C))\otimes M_{\dim(\alpha)}(\O(\G))$ is unitary;
\item If $T\in \Mor(\alpha\otimes \beta, \gamma)$, then
\begin{equation*}
\left[(m\otimes T)\circ\Sigma_{23}\right](a(\alpha)\otimes a(\beta)) = a(\gamma)\left[(m\otimes T)\circ\Sigma_{23}\right],
\end{equation*}
where $m\in \mathcal{L}(M_{N}(\C)\otimes M_{N}(\C), M_{N}(\C))$ is the matrix multiplication and $\Sigma_{23}$ is the flip between the second and third tensors;
\item The map $\eta : \lambda\in \C\mapsto \lambda\id \in M_{N}(\C)$ satisfies
\begin{equation*}
a(\varepsilon_{\G})(\id_{M_{N}(\C)}\otimes 1) = \id_{M_{N}(\C)}\otimes 1.
\end{equation*}
\end{itemize}
\end{de}

It was proven in \cite[Prop 2.7 ]{fima2015free} that there is a natural CQG-algebra structure on $A(\G)$ which turns it into a compact quantum group called the \emph{free wreath product of $\G$ by $\QA$} and denoted by $\G\wr_{\ast}\QA$. We are now ready for the first main result of this subsection.

\begin{thm}\label{thm:iteratedwreathsymetric}
For any integer $k$, there is an isomorphism of CQG-algebras
\begin{equation*}
\G_{N}(W^{(\circ, k+1)})\cong \G_{N}(W^{(\circ, k)})\wr_{\ast}\QA
\end{equation*}
\end{thm}

\begin{proof}
Recall that for a word $w$ on $\{\circ, \bullet\}$, $u^{w}$ denotes a representative of the corresponding irreducible unitary representation of $U_{N}^{+}$ acting on a Hilbert space $H_{w}$. We will fix an orthonormal basis $(f_{i})_{1\leqslant i\leqslant d_{w}}$ of that space in the sequel. Let now $w\in W^{(\circ, k)}$ and set (the second equality follows from the fact that $w$ starts with $\circ$ and ends with $\bullet$)
\begin{equation*}
b(w) = u^{\circ w\bullet} = u\otimes u^{w}\otimes \overline{u}.
\end{equation*}
This is a representation of $\G_{N}(W^{(\circ, k+1)})$ acting on $H\otimes H_{w}\otimes \overline{H}$. In order to use the definition of the free wreath product, we will see this carrier space as $M_{N}(\C)\otimes H_{w} \cong H\otimes \overline{H}\otimes H_{w}$ through the map
\begin{equation*}
\Phi_{w} : e_{i}\otimes \xi\otimes \overline{e}_{j}\mapsto e_{i}\otimes \overline{e}_{j}\otimes \xi
\end{equation*}
and accordingly set $a(w) = \Phi_{w} b(w)\Phi_{w}^{-1}$. Our goal is to check that these matrices satisfy the defining relations of Definition \ref{de:freewreath}. Observe for a start that, because $\Phi_{w}$ is a unitary map between two Hilbert spaces (with the Hilbert structure on $M_{N}(\C)$ coming from the trace), $a(w)$ is a unitary matrix.

To check the second condition, let $w_{1}, w_{2}, w_{3}\in W^{(\circ, k)}$ and let $T\in \Mor(u^{w_{1}}\otimes u^{w_{2}}, u^{w_{3}})$. For computational convenience, we will see $a(w)$ as a map $M_{N}(\C)\otimes H_{w}\to M_{N}(\C)\otimes H_{w}\otimes \O(\G)$. Then, for any $1\leqslant i, j, k, l\leqslant N$, $1\leqslant a \leqslant d_{w_{1}}$ and $1\leqslant a \leqslant d_{w_{2}}$,
\begin{align*}
& \left(\left[(m\otimes T)\circ\Sigma_{23}\right]\otimes\id\right)\left(a(w_{1})\otimes a(w_{2})\right)\left(e_{i}\otimes \overline{e}_{j}\otimes f_{a}\otimes e_{k}\otimes \overline{e}_{l}\otimes f_{b}\right) \\
& = \left(\left[(m\otimes T)\circ\Sigma_{23}\right]\otimes\id\right)\circ(\Phi_{w_{1}}\otimes \Phi_{w_{2}}\otimes \id)\left(b(w_{1})\otimes b(w_{2})\right)\left(e_{i}\otimes f_{a}\otimes \overline{e}_{j}\otimes e_{k}\otimes f_{b}\otimes \overline{e}_{l}\right) \\
& = \left(\left[(m\otimes T)\circ\Sigma_{23}\right]\otimes\id\right)\left(\sum_{i', j', k', l' = 1}^{N}\sum_{a', b' = 1}^{d_{w_{3}}}\Phi_{w_{1}}(e_{i'}\otimes f_{a'}\otimes \overline{e}_{j'})\otimes \Phi_{w_{2}}(e_{k'}\otimes f_{b'}\otimes \overline{e}_{l'})\otimes u_{ii'}u^{w_{1}}_{aa'}u_{jj'}^{*}u_{kk'}u^{w_{2}}_{bb'}u_{ll'}^{*}\right) \\
& = \left(\left[(m\otimes T)\circ\Sigma_{23}\right]\otimes\id\right)\left(\sum_{i', j', k', l' = 1}^{N}\sum_{a', b' = 1}^{d_{w_{3}}}e_{i'}\otimes \overline{e}_{j'}\otimes f_{a'} \otimes e_{k'}\otimes \overline{e}_{l'}\otimes f_{b'}\otimes u_{ii'}u^{w_{1}}_{aa'}u_{jj'}^{*}u_{kk'}u^{w_{2}}_{bb'}u_{ll'}^{*}\right) \\
& = \sum_{i', j', k', l' = 1}^{N}\sum_{a', b' = 1}^{d_{w_{3}}}\left(\left[(m\otimes T)\circ\Sigma_{23}\right]\otimes\id\right)\left(e_{i'}\otimes \overline{e}_{j'}\otimes e_{k'}\otimes \overline{e}_{l'}\otimes f_{a'}\otimes f_{b'}\otimes u_{ii'}u^{w_{1}}_{aa'}u_{jj'}^{*}u_{kk'}u^{w_{2}}_{bb'}u_{ll'}^{*}\right) \\
& = \sum_{i', j', l' = 1}^{N}\sum_{a', b' = 1}^{d_{w_{3}}}\left(e_{i'}\otimes \overline{e}_{l'}\otimes T(f_{a'}\otimes f_{b'})\otimes u_{ii'}u^{w_{1}}_{aa'}u_{jj'}^{*}u_{kj'}u^{w_{2}}_{bb'}u_{ll'}^{*}\right) \\
& = \delta_{kj}\sum_{i', l' = 1}^{N}\sum_{a', b' = 1}^{d_{w_{3}}}\left(e_{i'}\otimes \overline{e}_{l'}\otimes T(f_{a'}\otimes f_{b'})\otimes u_{ii'}u^{w_{1}}_{aa'}u^{w_{2}}_{bb'}u_{ll'}^{*}\right) \\
& = \delta_{kj}\sum_{i', l' = 1}^{N}\sum_{a', b' = 1}^{d_{w_{3}}}\left(e_{i'}\otimes \overline{e}_{l'}\otimes T(f_{a'}\otimes f_{b'})\otimes u_{ii'}u^{w_{1}}_{aa'}u^{w_{2}}_{bb'}u_{ll'}^{*}\right) \\
& = \delta_{kj}(\Phi_{w_{3}}\otimes \id)(\id\otimes T\otimes \id)\left(u\otimes u^{w_{1}}\otimes u^{w_{2}}\otimes \overline{u}\right)\left(e_{i}\otimes f_{a}\otimes f_{b}\otimes \overline{e}_{j}\right),
\end{align*}
while
\begin{align*}
& a(w_{3})\left(\left[(m\otimes T)\circ\Sigma_{23}\right]\otimes\id\right)\left(e_{i}\otimes \overline{e}_{j}\otimes f_{a}\otimes e_{k}\otimes \overline{e}_{l}\otimes f_{b}\right) = \delta_{kj}a(w_{3})\left(e_{i}\otimes \overline{e}_{l'}\otimes T(f_{a}\otimes f_{b})\right) \\
& = \delta_{kj}(\Phi_{w_{3}}\otimes \id)(u\otimes u^{w_{3}}\otimes \overline{u})\left(e_{i}\otimes T(f_{a}\otimes f_{b})\otimes \overline{e}_{l}\right). \\
\end{align*}
It then follows from the fact that $T\in \Mor(u^{w_{1}}\otimes u^{w_{2}}, u^{w_{3}})$ that $(m\otimes T)\circ\Sigma_{23}$ is an intertwiner. Moreover, the trivial representation is given by the empty word, and $b(\emptyset) = u\otimes \overline{u}$ fixes the vector
\begin{equation*}
\sum_{i=1}^{n}e_{i}\otimes 1_{\C}\otimes \overline{e}_{i},
\end{equation*}
which under the map $\Phi_{\emptyset}$ becomes $1_{M_{N}(\C)}\otimes 1$, as required.

By the universal property, there is therefore a surjective morphism of CQG-algebras
\begin{equation*}
\pi_{k} : \O(\G_{N}(W^{(\circ, k)}\wr_{\ast}\QA) \to \O(\G_{N}(W^{(\circ, k+1)}).
\end{equation*}
Moreover, $\pi_{k}$ induces at the level of fusion rings the map $\Psi$ of Proposition \ref{prop:fusionsymetric}, which has been proven to be injective. It is a well-known fact that this implies that $\pi_{k}$ is injective, but let us briefly sketch the argument: given a coefficient $x$ of an irreducible non-trivial representation $\alpha$ of $\O(\G_{N}(W^{(\circ, k)}\wr_{\ast}\QA)$, its image is a coefficient of $\Psi(\alpha)$, which is non-trivial and irreducible. As a consequence, applying the Haar state to $\Psi(\alpha)$ yields $0$; because coefficients of irreducible representations span $\O(\G_{N}(W^{(\circ, k)}\wr_{\ast}\QA)$, we conclude that its Haar state equals the Haar state of $\G_{N}(W^{(\circ, k+1)})$ composed with $\pi_{k}$, and the result then follows from the faithfulness of the Haar states.
\end{proof}

Note that the same proof works for $W^{(\bullet, k)}$. A straightforward induction then yields an explicit description of these compact quantum groups which therefore turn out to be isomorphic.

\begin{cor}
Both compact quantum groups $\G_{N}(W^{(\circ, k)})$ and $\G_{N}(W^{(\bullet, k)})$ are isomorphic to the $k$-th iterated free wreath product of $\QA$ by itself. 
\end{cor}

As for $\G_{N}(W^{(\circ, \infty)})$, it can be though of as an infinitely iterated free wreath product in the same way.

\subsubsection{The non-symetric case}

We conclude with the case of the compact quantum groups $\G(W^{(k)})$. We have not computed their fusion rings, but as we will see there is a direct way to describe them which does not require any prior knowledge of their representation theory. This will rely on two isomorphism results that we state separately for convenience. The first one concerns free wreath products with $\QA$ and is similar to what was used in the proof of Theorem \ref{thm:iteratedwreathsymetric}.

\begin{lem}\label{lem:wreath}
Let $(\G, v)$ be a compact matrix quantum group and let $\HH$ be a compact quantum group such that $\O(\HH)$ contains free copies of $\O(\G)$ and $\O(O_{N}^{+})$ with fundamental representation $u$. Then, the $*$-subalgebra $A\subset C(\HH)$ generated by the coefficients of $u\otimes v\otimes u$ is a CQG-algebra isomorphic to $\O(\G\wr_{\ast}\QA)$.
\end{lem}

\begin{proof}
For $\alpha\in \Irr(\G)$, set $a(\alpha) = u\otimes u^{\alpha}\otimes u$. Then, the proof of Theorem \ref{thm:iteratedwreathsymetric} goes along, noticing that because $u = \overline{u}$, we may see $a(\alpha)$ as a representation acting on $\C^{N}\otimes H_{\alpha}\otimes \overline{\C}^{N}$.
\end{proof}

Our second lemma concerns an alternative definition of the free complexification operation, which we first recall. Let us denote by $\mathbb{T}$ the compact group of complex numbers of modulus one, and by $z$ its fundamental representation given by the identity map. Given a compact matrix quantum group $(\G, v)$, its \emph{free complexification} is the quantum group $(\widetilde{\G}, vz)$, where $\O(\widetilde{\G})\subset \O(\G)\ast\O(\mathbb{T})$ is the $*$-subalgebra generated by the coefficients of the matrix $vz$.

\begin{lem}\label{lem:complexification}
Let $(\G, v)$ be a compact matrix quantum and let $\HH$ be a compact quantum group such that $\O(\HH)$ contains free copies of $\O(\G)$ and $\O(\mathbb{T})$ with fundamental representation $z$. Then, the $*$-subalgebra of $\O(\HH)$ generated by the coefficients of $zvz$ is a CQG-algebra isomorphic to $\O(\widetilde{\G})$.
\end{lem}

\begin{proof}
The proof relies on two remarks. First, if one considers the $*$-subalgebra generated by $uz^{2}$, then the same conclusion holds. Indeed, $z^{2}$ can be seen as the fundamental representation of $\O(\widehat{2\Z})\subset \O(\widehat{\Z}) = \O(\mathbb{T})$ and since $2\Z\cong \Z$ we are in the situation of the usual definition of the free complexification. Second, we claim that the map $x\mapsto \overline{z}xz$ is a CQG-algebra automorphism of $C(\HH)$. The fact that it is a $*$-automorphism is straightforward since it is the conjugation by a unitary. As for compatibility with the coproducts, this follows from the fact that $z$ is group-like, i.e. $\D(z) = z\otimes z$.

Summing up, if $A$ is the $*$-subalgebra generated by the coefficients of $zvz$, then it is a CQG-algebra isomorphic to the one generated by the coefficients of $\overline{z}(zvz)z = vz^{2}$, hence to the free complexification of $\G$.
\end{proof}

We are now ready for our second main result.

\begin{thm}\label{thm:descriptionquantumsubgroups}
We have $\G^{(1)} = U_{N}^{+}$, $\G^{(2)} \cong \left(\widehat{\Z}\wr_{\ast}\QA\right)^{\sim}$ and, for $k\geqslant 3$,
\begin{equation*}
\G^{(k)}\cong \left(\widetilde{\G}^{(k-2)}\wr_{\ast}\QA\right)^{\sim}.
\end{equation*}
\end{thm}

\begin{proof}
The first equality follows by definition. For the second one, observe that $\O(\G^{(2)})$ is generated by the coefficients of the representation $uzuz = (uzu)z$ in $\O(\O_{N}^{+}\ast \O(\mathbb{T}))$. It is therefore the free complexification of the quantum group corresponding to the CQG-algebra generated by the coefficients of $uzu$. But since $\O(\mathbb{T})$ and $\O(O_{N}^{+})$ are free inside $\O(U_{N}^{+})$ by definition, this is $\widehat{\Z}\wr_{\ast}\QA$ by Lemma \ref{lem:wreath}, hence the result.

If now $k\geqslant 3$, we observe that $\O(\G^{(k)})$ is generated by the coefficients of
\begin{equation*}
(uz)^{k} = \left(u\left[(zu)^{k-2}z\right]u\right)z = (uvu)z.
\end{equation*}
Moreover, $v = (zu)^{k-2}z = z(uz)^{k-2}z$ so that the CQG-algebra generated by its coefficients is isomorphic to $\O(\overline{\G^{(k-2)}})$ by Lemma \ref{lem:complexification} and we conclude again by Lemma \ref{lem:wreath} and Lemma \ref{lem:complexification}.
\end{proof}

\section{Connection to quantum graphs and quantum trees}\label{sec:quantumtrees}

Several definitions of quantum graphs have been given about ten years ago in the context of quantum information theory, amongst others see \cite{duan2012zero, weaver2012, musto2018} and then \cite{chirvasitu2022, wasilewski2023quantum, daws2024} for a comparison of the definitions. The literature on quantum graphs grows rapidly, see for instance \cite{brannan2019bigalois, brannan2022, gromada2021, matsuda2021,
weaver2021} in addition to the previously mentioned articles.

It was recently proven in \cite{brownlowe2023self} (see also \cite{de2023quantum, meunier2023quantum}) that the iterated free wreath product of $S_{N}^{+}$ with itself $k$ times can be identified with the quantum automorphism group of a rooted $N$-regular tree with depth $k$. Based on this, one may wonder whether the quantum groups $\G_{N}(W^{\circ, k})$ can be seen as quantum automorphism groups of regular quantum graphs. The base case $k = 1$ corresponds to the quantum automorphism group $\QA$ of $M_{N}(\C)$, hence it is natural to expect that we have to build a graph using the latter C*-algebra as a building block. However, we will start with a more general construction because it somehow makes the proofs simpler.

\subsection{Definition of quantum regular trees}

We now give a definition of quantum regular trees.
Let $(B, \psi)$ be a finite-dimensional C*-algebra equipped with a $\delta$-form $\psi$. Setting $B^{\otimes 0} = \C$, we will take
\begin{equation*}
B_{k} = \bigoplus_{i=0}^{k}B^{\otimes i}
\end{equation*}
as the non-commutative space underlying our quantum graph. As for the state, we would like to take the sum of the states $\psi^{\otimes i}$, but this will not be a $\delta$-form since $\psi^{\otimes i}$ is a $\delta^{i}$-form. Therefore, we need to renormalize them and set
\begin{equation*}
\delta_{k} = \sum_{i=0}^{k}\delta^{i} \text{ and } \psi_{k} = \frac{1}{\delta_{k}}\sum_{i=0}^{k}\delta^{i}\psi^{\otimes i},
\end{equation*}
with $\psi^{\otimes 0}$ being the identity on $\C$.

\begin{de}
For $0\leqslant i\leqslant k-1$, we define a linear map $A_{i} : B^{\otimes i}\to B^{\otimes i+1}$ by
\begin{equation*}
A_{i} : x\mapsto x\otimes 1.
\end{equation*}
We then set
\begin{equation*}
\mathcal{A}_{k} = \mathrm{Id} + \sum_{i=0}^{k-1} A_{i}.
\end{equation*}
\end{de}

\begin{lem}\label{lem:quantumregular}
The triple $\mathcal{G}_{k} = (B_{k}, \psi_{k}, \mathcal{A}_{k})$ is a quantum graph in the sense of \cite{wasilewski2023quantum}, i.e. $\mathcal{A}_{k}$ is completely positive and the following relation holds :
\begin{equation*}
m\circ(\mathcal{A}_{k}\otimes \mathcal{A}_{k})\circ m^{*} = \delta_{k}^{2}\mathcal{A}_{k}.
\end{equation*}
Moreover, if $B = \C^{N}$, then $\mathcal{G}_{k}$ is the rooted $N$-regular tree with depth $k$.
\end{lem}

\begin{proof}
It is clear that $A_{i}$ is completely positive for all $i$, hence $\mathcal{A}_{k}$ also is. Moreover, observe that $m(B^{\otimes i}\otimes B^{\otimes i'}) = 0$ if $i\neq i'$. Therefore, the range of $m^{*}$ is contained in 
\begin{equation*}
\bigoplus_{i=1}^{k}B^{\otimes i}\otimes B^{\otimes i}.
\end{equation*}
On the $i$-th summand, $\mathcal{A}_{k}\otimes \mathcal{A}_{k}$ is the sum of the three operators $\mathrm{Id}\otimes A_{i}$, $A_{i}\otimes \mathrm{Id}$ and $A_{i}\otimes A_{i}$. The first two have range contained in the kernel of $m$, hence we can focus on the last one. We compute
\begin{align*}
m\circ (A_{i}\otimes A_{i})\circ m^{*}(x) & = \sum m\left(A_{i}(x_{1})\otimes A_{i}(x_{2})\right) \\
& = \sum m\left(x_{1}\otimes 1\otimes x_{2}\otimes 1\right) \\
& = \sum x_{1}x_{2}\otimes 1 \\
& = \delta_{k}^{2}x\otimes 1 \\
& = \delta_{k}^{2}A_{i}(x).
\end{align*}

As for the last point of the statement, let us label the vertices of $\mathcal{G}_{k}$ by tuples $(i_{1}, \cdots, i_{p})$ corresponding to the basis vector $e_{i_{1}}\otimes \cdots\otimes e_{i_{p}}$ of $B^{\otimes p}$. Recalling that the unit of $\C^{N}$ is the all-one vector
\begin{equation*}
\xi = \sum_{i=1}^{N}e_{i},
\end{equation*}
we see that the $A_{i}$ component in the adjacency matrix means that the vertex $(i_{1}, \cdots, i_{p})$ is connected to the vertices $(i_{1}, \cdots, i_{p}, i)$ for all $1\leqslant i\leqslant N$. Moreover, the vertex corresponding to $B^{\otimes 0}$ is also connected to all the vertices labelled by $(i)$ for $1\leqslant i\leqslant N$. As a consequence, the corresponding graph is the rooted $N$-regular tree with depth $k$.
\end{proof}

The previous result makes it reasonable to call $\mathcal{G}_{k}$ the rooted $(B, \psi)$-regular quantum tree of depth $k$. For clarity, let us state this as a definition in the case which is of interest to us.

\begin{de}
The \emph{rooted tracial $M_{N}(\C)$-regular quantum tree of depth $k$} is the quantum graph given by $\mathcal{G}_{k} = (M_{N}(\C)^{\otimes k}, \mathrm{tr}_{k}, A_{\sigma_{k}})$.
\end{de}

The important observation is that $\G_{N}(W^{(\circ, k)})$ acts on $\mathcal{G}_{k}$. Indeed, let us consider first the quantum group $\G = U_{N}^{+}$ with its fundamental representation $V$. It acts on $M_{N}(\C)^{\otimes i}$ by conjugation, i.e. through the map
\begin{equation*}
\alpha_{i} : T\in M_{N}(\C)^{\otimes i}\mapsto V^{\otimes i}(1\otimes T)V^{\ast\otimes i}\in \O(\G)\otimes M_{N}(\C)^{\otimes i}.
\end{equation*}
This is trace-perserving because $V$ is unitary. Moreover, the first tensorand of $\alpha_{i}(T)$ is a product of $i$ coefficients of $V$ followed by a product of $i$ coefficients of $V^{*}$, hence it lies in $\O(\G_{N}(W^{\circ, i}))\subset \O(\G_{N}(W^{\circ, k}))$. Therefore, there is a unique action $\alpha : B_{k}\to \O(\G_{N}(W^{\circ, k}))\otimes B_{k}$ which coincides with $\alpha_{i}$ on the $i$-th component (and is the trivial action on $\C$). Let us now check the compatibility with the adjacency matrix. We have
\begin{align*}
\alpha_{i+1}\circ A_{i}(T) & = \alpha_{i}(T\otimes 1_{M_{N}(\C)}) \\
& = V^{\otimes (i+1)}(1\otimes T\otimes 1_{M_{N}(\C)})V^{\ast\otimes (i+1)} \\
& = \left[V^{\otimes i}(1\otimes T)V^{\ast i}\right]\otimes 1 \\
& = (\id\otimes A_{i})\left(V^{\otimes i}(1\otimes T)V^{\ast\otimes i}\right)\\
& = (\id\otimes A_{i})\circ\alpha_{i}(T)
\end{align*}
hence $\G_{N}(W^{(\circ, k)})$ acts on $\mathcal{G}_{k}$.

We conclude with a converse to this computation.

\begin{thm}\label{thm:quantumtrees}
The compact quantum group $\G_{N}(W^{(\circ, k)})$ is the quantum automorphism group of the rooted $M_{N}(\C)$-regular quantum tree of depth $k$.
\end{thm}

\begin{proof}
We will prove the result by induction on $k$. For $k = 1$ this is already known. Assuming now the result for some $k\geqslant 1$, we have a surjective $*$-homomorphism
\begin{align*}
\Psi : \O(\QAut(\mathcal{G}_{k+1})) & \to \O(\G_{N}(W^{(\circ, k+1}))) \\
& = \O(\G_{N}(W^{(\circ, k)})\wr_{\ast}\QAut(M_{N}(\C))) \\
& = \O(\QAut(\mathcal{G}_{k}\wr_{\ast}\QAut(M_{N}(\C))))
\end{align*}
following from the universality of $\O(\QAut(\mathcal{G}_{k+1}))$ and the fact that the right-hand side acts on $\mathcal{G}_{k+1}$. Our strategy will be to construct an inverse to that map using the universal property of free wreath products. This requires finding homomorphic images of both $\O(\QAut(\mathcal{G}_{k}))$ and $\O(\QAut(M_{N}(\C)))$ inside $\O(\QAut(\mathcal{G}_{k+1}))$. To lighten notations, we will write $B$ for $M_{N}(\C)$ in the sequel.

Observe that the range of $\mathcal{A}_{k+1}$ is by definition stable under the action. Moreover, it is exactly the space of all elements such that the last tensorand of each component is $1$, and therefore it is a sub-C*-algebra isomorphic to $B_{k}$. The restriction of $\mathcal{A}_{k+1}$ to it is by definition $\mathcal{A}_{k}$ and the restriction of $\psi_{k+1}$ is a multiple of $\psi_{k}$. In other words, this yields a quantum graph isomorphic to $\mathcal{G}_{k}$ which is stable under the action. As a consequence, there is a $*$-homomorphism
\begin{equation*}
\pi_{1} : \O(\QAut(\mathcal{G}_{k}))\to \O(\QAut(\mathcal{G}_{k+1}))
\end{equation*}
which has the following property which we record for later use : if $b$ is in the range of $\mathcal{A}_{k+1}$, then $\alpha(b)\in \pi(\O(\QAut(\mathcal{G}_{k})))\otimes B_{k+1}$.

Let us now consider $B' = 0\oplus B\oplus 0\oplus \cdots\oplus 0$, which is a sub-C*-algebra of $B_{k+1}$ isomorphic to $B$. We claim that is it also stable under the action. Indeed, observe that $A_{1}^{*}A_{1}$ commutes with the action and is the projection onto $B'$. This means that $\QAut(\mathcal{G}_{k+1})$ acts on $B$. Since that action preserves $\psi$, which is a multiple of the trace, there is a surjective $*$-homomorphism
\begin{equation*}
\pi_{2} : \O(\QAut(M_{N}(\C)))\to \O(\QAut(\mathcal{G}_{k+1})).
\end{equation*}

We can now gather the two $*$-homomorphisms that we found to produce generators satisfying the defining relation of Definition \ref{de:freewreath}. If $u$ is an irreducible representation of $\QAut(\mathcal{G}_{k})$ and if $v$ denotes the fundamental representation of $\QAut(M_{N}(\C))$, we set
\begin{equation*}
a(u) = \pi_{1}(u)\otimes\pi_{2}(v).
\end{equation*}
These are unitary representations satisfying the conditions of Definition \ref{de:freewreath}. The only thing which is not clear is that their coefficients generate $\O(\QAut(\mathcal{G}_{k+1}))$. To see this, we will prove that if $A\subset \O(\QAut(\mathcal{G}_{k+1}))$ is the $*$-subalgebra generated by the coefficients of $a(u)$ for $u$ ranging through all irreducible representations of $\O(\QAut(\mathcal{G}_{k}))$, then $\alpha(B_{k+1})\subset A\otimes B_{k+1}$. To see this, let us consider an arbitrary element $x\in B_{k+1}$. If $x_{i}$ is its component in $B^{\otimes i}$, then it is enough to prove that $\alpha(x_{i})\subset A\otimes B_{k+1}$. Now, we can write
\begin{equation*}
x_{i} = b_{1}\otimes \cdots\otimes b_{i} = (b_{1}\otimes \cdots\otimes b_{i-1}\otimes 1)(1\otimes \cdots\otimes 1\otimes b_{i}) = b_{i}'b_{i}''.
\end{equation*}
Observe that $b_{i}'$ is in the range of $\mathcal{A}_{k+1}$, so that $\alpha(b_{i}')\subset A\otimes B_{k+1}$. As for $b_{i}''$, it can be written as $A_{i-1}\circ A_{i-2}\circ\cdots\circ A_{1}(b)$, hence
\begin{equation*}
\alpha(b_{i}'') = (\id\otimes A_{i-1}\circ A_{i-2}\circ\cdots\circ A_{1})\circ\alpha(b) \in A\otimes B_{k+1}
\end{equation*}
since the coefficients in the left tensorand of $\alpha(b)$ are by definition coefficients of $\pi_{2}(v)$. This completes the argument, proving that by the universal property of free wreath products, there is a surjective $*$-homomorphism
\begin{equation*}
\Phi : \O(\QAut(\mathcal{G}_{k})\wr_{\ast}M_{N}(\C)) \to \O(\QAut(\mathcal{G}_{k+1})).
\end{equation*}

Let now $w$ denote the fundamental representation of $\QAut(\mathcal{G}_{k})$ and $a(w)$ its canonical image as a representation of $\O(\QAut(\mathcal{G}_{k})\wr_{\ast}M_{N}(\C))$. Its coefficients are sent by $\Phi$ to the coefficients in the left tensorand of $\alpha(x)$ for $x\in B_{k}$. These are in turn sent to the corresponding coefficients for the action of $\G_{N}(W^{(\circ, k)})$ which are exactly the coefficients of $V^{\otimes k}\otimes V^{\ast\otimes k}$. Therefore, $\Psi\circ\Phi$ is injective on the $*$-subalgebra generated by the coefficients of $a(u)$ for all irreducible representations $u$ of $\QAut(\mathcal{G}_{k})$. Similarly, the coefficients of $v$ are sent by $\Psi\circ\Phi$ to the corresponding coefficients of $V$. We conclude that the fundamental representation is sent to the fundamental representation, which implies that $\Psi\circ\Phi$ is injective, so that $\Phi$ also is. Since it is surjective by construction, the proof is complete.
\end{proof}

\begin{rem}
It seems plausible that for an arbitrary finite quantum space $(B, \psi)$ the isomorphism
\begin{equation*}
\QAut(\mathcal{G}_{k+1}) \simeq \QAut(\mathcal{G}_{k})\wr_{\ast}\QAut(B, \psi)
\end{equation*}
holds. This could then be seen as an extension of the result of J. Bichon \cite{bichon2004free} for disjoint unions of transitive graphs. Indeed, for a classical rooted $N$-regular tree of depth $k+1$, the root is fixed by the quantum automorphism group because it is the only vertex with degree $N+1$, hence the quantum automorphism group is just the quantum automorphism group of the disjoint union of $N$ rooted $N$-regular trees of depth $k$.
\end{rem}

\bibliographystyle{amsplain}
\bibliography{quantum.bib}

\providecommand{\bysame}{\leavevmode\hbox to3em{\hrulefill}\thinspace}
\providecommand{\MR}{\relax\ifhmode\unskip\space\fi MR }
\providecommand{\MRhref}[2]{%
  \href{http://www.ams.org/mathscinet-getitem?mr=#1}{#2}
}
\providecommand{\href}[2]{#2}
\begin{thebibliography}{10}

\bibitem{banica1997groupe}
T.~Banica, \emph{{Le groupe quantique compact libre $U(n)$}}, Comm. Math. Phys.
  \textbf{190} (1997), no.~1, 143--172.

\bibitem{banica1999symmetries}
\bysame, \emph{{Symmetries of a generic coaction}}, Math. Ann. \textbf{314}
  (1999), no.~4, 763--780.

\bibitem{banica2007note}
\bysame, \emph{{A note on free quantum groups}}, Ann. Math. Blaise Pascal
  \textbf{15} (2008), 135--146.

\bibitem{banica2010classification}
T.~Banica, S.~Curran, and R.~Speicher, \emph{{Classification results for easy
  quantum groups}}, Pacific J. Math. \textbf{247} (2010), no.~1, 1--26.

\bibitem{banica2009liberation}
T.~Banica and R.~Speicher, \emph{{Liberation of orthogonal Lie groups}}, Adv.
  Math. \textbf{222} (2009), no.~4, 1461--1501.

\bibitem{banica2010invariants}
T.~Banica and R.~Vergnioux, \emph{{Invariants of the half-liberated orthogonal
  group}}, Ann. Inst. Fourier \textbf{60} (2010), no.~6, 2137--2164.

\bibitem{bichon2004free}
J.~Bichon, \emph{Free wreath product by the quantum permutation group}, Algebr.
  Represent. Theory \textbf{7} (2004), no.~4, 343--362.

\bibitem{brannan2019bigalois}
M.~Brannan, A.~Chirvasitu, K.~Eifler, S.~Harris, V.~Paulsen, X.~Su, and
  M.~Wasilewski, \emph{{Bigalois extensions and the graph isomorphism game}},
  Comm. Math. Phys. (2019), 1--33.

\bibitem{brannan2022}
Michael Brannan, Kari Eifler, Christian Voigt, and Moritz Weber, \emph{Quantum
  {C}untz-{K}rieger algebras}, Trans. Amer. Math. Soc. Ser. B \textbf{9}
  (2022), 782--826. \MR{4494623}

\bibitem{brownlowe2023self}
N.~Brownlowe and D.~Robertson, \emph{Self-similar quantum groups}, arXiv
  preprint (2023).

\bibitem{chirvasitu2022}
Alexandru Chirvasitu and Mateusz Wasilewski, \emph{Random quantum graphs},
  Trans. Amer. Math. Soc. \textbf{375} (2022), no.~5, 3061--3087. \MR{4402656}

\bibitem{cirio2014connected}
L.~Cirio, A.~D'Andrea, C.~Pinzari, and S.~Rossi, \emph{Connected components of
  compact matrix quantum groups and finiteness conditions}, J. Funct. Anal.
  \textbf{267} (2014), no.~9, 3154--3204.

\bibitem{daws2024}
Matthew Daws, \emph{Quantum graphs: {D}ifferent perspectives, homomorphisms and
  quantum automorphisms}, Comm. Amer. Math. Soc. \textbf{4} (2024), 117--181.
  \MR{4706978}

\bibitem{de2023quantum}
Josse Van~Dobben De~Bruyn, Prem~Nigam Kar, David~E Roberson, Simon Schmidt, and
  Peter Zeman, \emph{Quantum automorphism groups of trees}, arXiv preprint
  arXiv:2311.04891 (2023).

\bibitem{dijkhuizen1994CQG}
M.S. Dijkhuizen and T.H. Koornwinder, \emph{{CQG algebras : a direct algebraic
  approach to compact quantum groups}}, Lett. Math. Phys. \textbf{32} (1994),
  315--330.

\bibitem{duan2012zero}
R.~Duan, S.~Severini, and A.~Winter, \emph{{Zero-error communication via
  quantum channels, noncommutative graphs, and a quantum Lov{\'a}sz number}},
  IEEE Trans. Inform. Theory \textbf{59} (2012), no.~2, 1164--1174.

\bibitem{fima2015free}
P.~Fima and L.~Pittau, \emph{{The free wreath product of a compact quantum
  group by a quantum automorphism group}}, J. Funct. Anal. \textbf{27} (2016),
  no.~7, 1996--2043.

\bibitem{freslon2013fusion}
A.~Freslon, \emph{{Fusion (semi)rings arising from quantum groups}}, J. Algebra
  \textbf{417} (2014), 161--197.

\bibitem{freslon2014partition}
\bysame, \emph{{On the partition approach to Schur-Weyl duality and free
  quantum groups -- with an appendix by A. Chirvasitu}}, Transform. Groups
  \textbf{22} (2017), no.~3, 705--751.

\bibitem{freslon2023compact}
\bysame, \emph{Compact matrix quantum groups and their combinatorics}, {LMS
  Student Texts in Mathematics}, {Cambridge University Press}, 2023.

\bibitem{freslon2021tannaka}
A.~Freslon, F.~Taipe, and S.~Wang, \emph{{Tannaka-Krein reconstruction and
  ergodic actions of easy quantum groups}}, Comm. Math. Phys. (2022).

\bibitem{freslon2013representation}
A.~Freslon and M.~Weber, \emph{{On the representation theory of partition
  (easy) quantum groups}}, J. Reine Angew. Math. \textbf{720} (2016), 155--197.

\bibitem{gromada2021}
Daniel Gromada, \emph{{Some examples of quantum graphs}}, arXiv preprint
  arXiv:2109.13618 (2021).

\bibitem{matsuda2021}
Junichiro Matsuda, \emph{{Classification of quantum graphs on M2 and their
  quantum automorphism groups}}, arXiv preprint arXiv:2110.09085 (2021).

\bibitem{meunier2023quantum}
Paul Meunier, \emph{Quantum properties of $\mathcal{F}$-cographs}, arXiv
  preprint arXiv:2312.01516 (2023).

\bibitem{musto2018}
Benjamin Musto, David Reutter, and Dominic Verdon, \emph{A compositional
  approach to quantum functions}, J. Math. Phys. \textbf{59} (2018), no.~8,
  081706, 42. \MR{3849575}

\bibitem{neshveyev2014compact}
S.~Neshveyev and L.~Tuset, \emph{{Compact quantum groups and their
  representation categories}}, Cours Sp\'ecialis\'es, vol.~20, Société
  Mathématique de France, 2013.

\bibitem{raum2014combinatorics}
S.~Raum and M.~Weber, \emph{{The combinatorics of an algebraic class of easy
  quantum groups}}, Infin. Dimens. Anal. Quantum Probab. Relat. Top.
  \textbf{17} (2014), no.~3, 1450016.

\bibitem{raum2013easy}
\bysame, \emph{{Easy quantum groups and quantum subgroups of a semi-direct
  product quantum group}}, J. Noncommut. Geom. \textbf{9} (2015), no.~4,
  1261--1293.

\bibitem{raum2013full}
\bysame, \emph{{The full classification of orthogonal easy quantum groups}},
  Comm. Math. Phys. \textbf{341} (2016), no.~3, 751--779.

\bibitem{tarrago2015unitary}
P.~Tarrago and M.~Weber, \emph{{Unitary easy quantum groups : the free case and
  the group case}}, Int. Math. Res. Not. \textbf{18} (2017), 5710--5750.

\bibitem{tarrago2018classification}
\bysame, \emph{{The classification of tensor categories of two-colored
  noncrossing partitions}}, J. Combin. Theory Ser. A \textbf{154} (2018),
  464--506.

\bibitem{timmermann2008invitation}
T.~Timmermann, \emph{{An invitation to quantum groups and duality. From Hopf
  algebras to multiplicative unitaries and beyond}}, EMS Textbooks in
  Mathematics, European Mathematical Society, 2008.

\bibitem{vaes2007boundary}
S.~Vaes and R.~Vergnioux, \emph{{The boundary of universal discrete quantum
  groups, exactness and factoriality}}, Duke Math. J. \textbf{140} (2007),
  no.~1, 35--84.

\bibitem{vergnioux2004k}
R.~Vergnioux, \emph{K-amenability for amalgamated free products of amenable
  discrete quantum groups}, J. Funct. Anal. \textbf{212} (2004), no.~1,
  206--221.

\bibitem{vergnioux2013k}
R.~Vergnioux and C.~Voigt, \emph{{The K-theory of free quantum groups}}, Math.
  Ann. \textbf{357} (2013), no.~1, 355--400.

\bibitem{wang1995free}
S.~Wang, \emph{Free products of compact quantum groups}, Comm. Math. Phys.
  \textbf{167} (1995), no.~3, 671--692.

\bibitem{wang1998quantum}
\bysame, \emph{Quantum symmetry groups of finite spaces}, Comm. Math. Phys.
  \textbf{195} (1998), no.~1, 195--211.

\bibitem{wasilewski2023quantum}
M.~Wasilewski, \emph{{On quantum Cayley graphs}}, ArXiv preprint (2023).

\bibitem{weaver2012}
Nik Weaver, \emph{Quantum relations}, Mem. Amer. Math. Soc. \textbf{215}
  (2012), no.~1010, v--vi, 81--140. \MR{2908249}

\bibitem{weaver2021}
\bysame, \emph{Quantum graphs as quantum relations}, J. Geom. Anal. \textbf{31}
  (2021), no.~9, 9090--9112. \MR{4302212}

\bibitem{weber2012classification}
M.~Weber, \emph{{On the classification of easy quantum groups -- The
  nonhyperoctahedral and the half-liberated case}}, Adv. Math. \textbf{245}
  (2013), no.~1, 500--533.

\bibitem{weber2015introduction}
\bysame, \emph{{Introduction to compact (matrix) quantum groups and
  Banica-Speicher (easy) quantum groups}}, Proceedings of the Indian Acamedy of
  Sciences. Mathematical Sciences \textbf{127} (2017), no.~5, 881--933.

\bibitem{woronowicz1988tannaka}
S.L. Woronowicz, \emph{{Tannaka-Krein duality for compact matrix pseudogroups.
  Twisted $SU(N)$ groups}}, Invent. Math. \textbf{93} (1988), no.~1, 35--76.

\bibitem{woronowicz1995compact}
\bysame, \emph{{Compact quantum groups}}, Sym{\'e}tries quantiques (Les
  Houches, 1995) (1998), 845--884.

\end{thebibliography}

\end{document}